\documentclass[11pt,reqno]{amsart} 

\usepackage{amsthm,amsfonts,amscd,amsmath}
\usepackage[hidelinks]{hyperref} 
\usepackage{url}
\usepackage{xcolor}
\hypersetup{
    colorlinks,
    linkcolor={red!50!black},
    citecolor={blue!50!black},
    urlcolor={blue!80!black}
}

\usepackage[normalem]{ulem}
\usepackage{graphicx} 
\usepackage{datetime} 
\usepackage{cancel}
\usepackage{caption,subcaption}
\numberwithin{figure}{section}
\numberwithin{table}{section}

\usepackage{stmaryrd} 
\usepackage{framed} 

\let\citep\cite

\newcommand{\undersetbrace}[2]{\underset{\displaystyle{#1}}{\underbrace{#2}}}

\newcommand{\cf}{\textit{cf.\ }} 
\newcommand{\Iverson}[1]{\ensuremath{\left[#1\right]_{\delta}}} 

\newcommand{\floor}[1]{\left\lfloor #1 \right\rfloor} 
\newcommand{\ceiling}[1]{\left\lceil #1 \right\rceil} 
\newcommand{\e}[1]{e\left(#1\right)} 

\usepackage{upgreek}

\renewcommand{\chi}{\upchi}

\DeclareMathOperator{\ds}{ds} 
\DeclareMathOperator{\Id}{Id}

\title[Factorization Theorems for Relatively Prime Divisor Sums]{
       Factorization Theorems for 
       Relatively Prime Divisor Sums, GCD Sums and 
       Generalized Ramanujan Sums 
} 
\author[H. Mousavi and M. D. Schmidt]{
         Hamed Mousavi \\ 
         Maxie D. Schmidt 
} 

\email{\href{mailto:hmousavi6@gatech.edu}{hmousavi6@gatech.edu} \\ 
     \href{mailto:maxieds@gmail.com}{maxieds@gmail.com} \\ 
       \href{mailto:mschmidt34@gatech.edu}{mschmidt34@gatech.edu}
}
\address{Georgia Institute of Technology \\ 
         School of Mathematics \\ 
         117 Skiles Building \\ 
         686 Cherry Street NW \\ 
         Atlanta, GA 30332 \\ 
         USA
} 

\date{\today} 


\allowdisplaybreaks 

\theoremstyle{plain} 
\newtheorem{theorem}{Theorem}

\newtheorem{prop}[theorem]{Proposition}
\newtheorem{lemma}[theorem]{Lemma}
\newtheorem{cor}[theorem]{Corollary}
\numberwithin{theorem}{section}

\theoremstyle{definition} 
\newtheorem{example}[theorem]{Example}
\newtheorem{remark}[theorem]{Remark}
\newtheorem{definition}[theorem]{Definition}
\newtheorem{notation}[theorem]{Notation}

\usepackage{longtable}
\usepackage{arydshln} 
\usepackage[symbols,nogroupskip,shortcuts,nomain,automake,ucmark]{glossaries-extra}
\usepackage{glossary-mcols}

\newglossarystyle{glossstyle}{%
 {\begin{longtable}{lp{\glsdescwidth}}}%
 {\end{longtable}}%
 \setlength{\parskip}{3.5pt}

 \renewcommand*{\glsgroupheading}[1]{}%

  \renewcommand{\glossarymark}[1]{}
}
\setglossarystyle{glossstyle}
\glsaddall[types={symbols}]
\makeglossaries

\newglossaryentry{fCvlg}{
    symbol={\ensuremath{\ast; f \ast g}},
    sort={fg},
    description={The Dirichlet convolution of $f$ and $g$, $f \ast g(n) := \sum\limits_{d|n} f(d) g(n/d)$, for $n \geq 1$. This symbol for the discrete convolution of two arithmetic functions is the only notion of convolution of functions we employ within the article.},
    type={symbols},
    name={Dirichlet convolution}
    }
\newglossaryentry{coeffExtraction}{
    symbol={\ensuremath{[q^n] F(q)}},
    sort={coeffExtraction},
    description={The coefficient of $q^n$ in the power series expansion of $F(q)$ about zero.},
    type={symbols},
    name={Series coefficient extraction}
    }
\newglossaryentry{MoebiusMuFunc}{
    symbol={\ensuremath{\mu(n)}},
    sort={MoebiusMuFunc},
    description={The M\"obius function.},
    type={symbols},
    name={M\"obius function}
    }
\newglossaryentry{mGCDn}{
    symbol={\ensuremath{\operatorname{gcd}(m, n); (m,n)}},
    sort={mGCDn},
    description={The greatest common divisor of $m$ and $n$. Both notations for the GCD are used 
    interchangably within the article. },
    type={symbols},
    name={Greatest common divisor}
    }
\newglossaryentry{IversonI}{
    symbol={\ensuremath{\Iverson{n=k}}},
    sort={IversonI},
    description={Synonym for $\delta_{n,k}$ which is one if and only if $n = k$, and zero otherwise.},
    type={symbols},
    name={Iverson's convention}
    }
\newglossaryentry{IversonII}{
    symbol={\ensuremath{\Iverson{\mathtt{cond}}}},
    sort={IversonII},
    description={For a boolean-valued \texttt{cond}, $\Iverson{\mathtt{cond}}$ evaluates to one precisely when \texttt{cond} is true, and zero otherwise.},
    type={symbols},
    name={Iverson's convention}
    }
\newglossaryentry{yfn}{
    symbol={\ensuremath{y_f(n)}},
    sort={yfn},
    description={The function $y_f(n)$ denotes the Dirichlet inverse of the function 
    $h(n) := f(n) \phi(n) / n^2$ where \glssymbol{phin} is Euler's totient function and 
    $f$ is any invertible arithmetic function such that $f(1) \neq 0$. This function is used to 
    express the result in Corollary \ref{cor_exact_formula_forn_gn} on page 
    \pageref{cor_exact_formula_forn_gn}. 
    A special case, denoted by $y(n)$, corresponding to the case where 
    $f(n) \equiv n$ is employed in stating Corollary \ref{cor_Mertens_function_v2} in Section 
    \ref{subSection_DFT_exps}.},
    type={symbols},
    name={A special case Dirichlet inverse function}
    }
\newglossaryentry{tnk}{
    symbol={\ensuremath{t_{n,k}}},
    sort={tnk},
    description={Matrix sequence involved in the generating function expansions of the type I sums defined on page \pageref{eqn_intro_two_sum_fact_types_def_v1} as 
    \[
    T_{f}(x) = [q^x]\left(\frac{1}{(q; q)_{\infty}} \sum_{n \geq 2} \sum_{k=1}^n t_{n,k} f(k) \cdot q^n + 
     f(1) \cdot q\right)
    \]},
    type={symbols},
    name={Matrix coefficients}
    }
\newglossaryentry{tnkInv}{
    symbol={\ensuremath{t_{n,k}^{(-1)}}},
    sort={tnkInv},
    description={Inverse matrix of the sequence \glssymbol{tnk} that is defined on page \pageref{theorem_snk_snkinv_seq_formulas_v1}.},
    type={symbols},
    name={Inverse matrix coefficients}
    }
\newglossaryentry{chikn}{
    symbol={\ensuremath{\chi_{1,k}(n)}},
    sort={chikn},
    description={The principal Dirichlet character modulo $k$, i.e., the indicator function of the natural numbers which are relatively prime for $n,k \geq 1$, $\chi_{1,k}(n) = \Iverson{(n,k)=1}$.},
    type={symbols},
    name={Principal Dirichlet character modulo $k$}
    }
\newglossaryentry{dn}{
    symbol={\ensuremath{d(n)}},
    sort={dn},
    description={The ordinary divisor function, $d(n) := \sum_{d|n} 1$.},
    type={symbols},
    name={Divisor function}
    }
\newglossaryentry{epsilonN}{
    symbol={\ensuremath{\varepsilon(n)}},
    sort={epsilonN},
    description={The multiplicative identity with respect to Dirichlet convolution, $\varepsilon(n) = \delta_{n,1}$.},
    type={symbols},
    name={Dirichlet multiplicative identity}
    }
\newglossaryentry{munk}{
    symbol={\ensuremath{\mu_{n,k}}},
    sort={munk},
    description={Matrix sequence defined on page \pageref{prof_inversion_formula}. The invertible sequence is an analog to the role of the M\"obius function in M\"obius inversion. In this case these inversion coefficients are defined such that 
    \[
    g(n) = \sum_{\substack{d=1 \\ (d, n)=1}}^n 
     f(d) \quad\iff\quad f(n) = \sum_{d=1}^n g(d+1) \mu_{n,d}.
    \]
    See Proposition \ref{prof_inversion_formula} and Section 
    \ref{Section_Sum_TypeI} for the relation of this sequence (and its inverse) to the 
    factorizations of type I sums.
    },
    type={symbols},
    name={Matrix coefficients}
    }
\newglossaryentry{munkInv}{
    symbol={\ensuremath{\mu_{n,k}^{(-1)}}},
    sort={munkInv},
    description={Inverse matrix sequence of \glssymbol{munk} defined on page \pageref{prof_inversion_formula}.},
    type={symbols},
    name={Inverse matrix coefficients}
    }
\newglossaryentry{unk}{
    symbol={\ensuremath{u_{n,k}(f, w)}},
    sort={unk},
    description={Matrix sequence defined on page \pageref{eqn_intro_two_sum_fact_types_def_v1} in the expansion of the generating functions for the type II sums as 
    \[
    g(x) = [q^x]\left(\frac{1}{(q; q)_{\infty}} \sum_{n \geq 2} \sum_{k=1}^n u_{n,k}(f, w) \left[ 
     \sum_{m=1}^k L_{f,g,m}(k) w^m\right] \cdot q^n\right),\ w \in \mathbb{C}. 
    \]},
    type={symbols},
    name={Matrix coefficients}
    }
\newglossaryentry{unkInv}{
    symbol={\ensuremath{u_{n,k}^{(-1)}(f, w)}},
    sort={unkInv},description={Inverse matrix terms of the sequence \glssymbol{unk} defined on page \pageref{eqn_intro_two_sum_fact_types_def_v1}.},
    type={symbols},
    name={Inverse matrix coefficients}
    }
\newglossaryentry{DTFT}{
    symbol={\ensuremath{\operatorname{DTFT}[f](k)}},
    sort={DTFT},
    description={The discrete time Fourier transform (DTFT) of $f$ at $k$, also denoted by $F[k]$.},
    type={symbols},
    name={DTFT}
    }
\newglossaryentry{DFT}{
    symbol={\ensuremath{\operatorname{DFT}[f](k)}},
    sort={DTFT},
    description={The discrete Fourier transform (DFT) of $f$ at $k$. We use this transformation in Section \ref{subSection_DFT_exps} of the article.},
    type={symbols},
    name={DTFT}
    }
\newglossaryentry{phin}{
    symbol={\ensuremath{\phi(n)}},
    sort={phin},
    description={Euler's classical totient function, $\phi(n) := \sum\limits_{\substack{1 \leq d \leq n \\ (d,n) = 1}} 1$.},
    type={symbols},
    name={Euler totient function}
    }
\newglossaryentry{phikn}{
    symbol={\ensuremath{\phi_k(n)}},
    sort={phikn},
    description={Generalized totient function, $\phi_k(n) := \sum\limits_{\substack{1 \leq d \leq n \\ (d,n) = 1}} d^k$.},
    type={symbols},
    name={Generalized totient function}
    }
\newglossaryentry{Tfx}{
    symbol={\ensuremath{T_f(x)}},
    sort={Tfx},
    description={The type I sum over an arithmetic function $f$, $T_f(n) := \sum\limits_{\substack{d \leq x \\ (d,x)=1}} f(d)$.},
    type={symbols},
    name={Type I sum}
    }
\newglossaryentry{Lfgkx}{
    symbol={\ensuremath{L_{f,g,k}(x)}},
    sort={Lfgkx},
    description={The type II Anderson-Apostol sum over the arithmetic functions $f,g$, $L_{f,g,k}(x) := \sum\limits_{d|(k,x)} f(d) g(x/d)$.},
    type={symbols},
    name={Type II sum}
    }
\newglossaryentry{pn}{
    symbol={\ensuremath{p(n)}},
    sort={pn},
    description={The partition function generated by $p(n) = [q^n] \prod\limits_{n \geq 1} (1-q^n)^{-1}$.},
    type={symbols},
    name={Partition function $p$}
    }
\newglossaryentry{Mx}{
    symbol={\ensuremath{M(x)}},
    sort={Mx},
    description={The Mertens function which is the summatory function over $\mu(n)$, $M(x) := \sum\limits_{n \leq x} \mu(n)$.},
    type={symbols},
    name={Mertens function}
    }
\newglossaryentry{Phinz}{
    symbol={\ensuremath{\Phi_n(z)}},
    sort={Phinz},
    description={The $n^{th}$ cyclotomic polynomial in $z$ defined by $\Phi_n(z) := \prod\limits_{\substack{1 \leq k \leq n \\ (k,n)=1}} (z-e^{2\pi\imath k/n})$.},
    type={symbols},
    name={Cyclotomic polynomial}
    }
\newglossaryentry{cqn}{
    symbol={\ensuremath{c_q(n)}},
    sort={cqn},
    description={Ramanujan's sum, $c_q(n) := \sum\limits_{d|(q,n)} d \mu(q/d)$.},
    type={symbols},
    name={Ramanujan's sum}
    }
\newglossaryentry{SigmaSOD}{
    symbol={\ensuremath{\sigma_{\alpha}(n)}},
    sort={SigmaSOD},
    description={The generalized sum-of-divisors function, $\sigma_{\alpha}(n) := \sum\limits_{d|n} d^{\alpha}$, for any $n \geq 1$ and $\alpha \in \mathbb{C}$.},
    type={symbols},
    name={Generalized sum-of-divisors function}
    }
\newglossaryentry{Zetas}{
    symbol={\ensuremath{\zeta(s)}},
    sort={Zetas},
    description={The Riemann zeta function, defined by $\zeta(s) := \sum\limits_{n \geq 1} n^{-s}$ when $\Re(s) > 1$, and by analytic continuation to the entire complex plane with the exception of a simple pole at $s = 1$.},
    type={symbols},
    name={Riemann zeta function}
    }
\newglossaryentry{Expex}{
    symbol={\ensuremath{e(x)}},
    sort={Expex},
    description={The complex exponential function, $e(x) := \exp(2\pi\imath \cdot x)$.},
    type={symbols},
    name={Complex exponential function shorthand}
    }
\newglossaryentry{snk}{
    symbol={\ensuremath{s_{n,k}}},
    sort={snk},
    description={Matrix coefficients in Lambert series type factorizations defined on page \pageref{page_snk_LSeriesFactMatrixCoeffs_ref}. These coefficients are defined precisely as the coefficients of the generating function $[q^n] (q; q)_{\infty} q^k / (1-q^k)$ for $k \geq 1$ where 
    \glssymbol{qPochInfty} is the infinite $q$-Pochhammer symbol.},
    type={symbols},
    name={Lambert series factorization matrix coefficients}
    }
\newglossaryentry{fhatn}{
    symbol={\ensuremath{\hat{f}(n)}},
    sort={fhatn},
    description={A shorthand notation for scaled arithmetic function terms $\hat{f}(n) := w^n / (w^n-1) f(n)$ for some indeterminate $w$. The notation is defined on page \pageref{prop_Pjwt_poly_matrix_exp}.},
    type={symbols},
    name={Auxiliary scaled functions}
    }
\newglossaryentry{fpmn}{
    symbol={\ensuremath{f_{\pm}(n)}},
    sort={fpmn},
    description={For any arithmetic function $f$, we define 
    $f_{\pm}(n) = f(n) \Iverson{n > 1} - f(1) \Iverson{n = 1}$, i.e., 
    the function that has identical values as $f$ for all $n \geq 2$, and 
    whose initial value is $f_{\pm}(1) := -f(1)$ when $n = 1$.},
    type={symbols},
    name={Case-wise modified functions}
    }
\newglossaryentry{uhatnkfw}{
    symbol={\ensuremath{\hat{u}_{n,k}(f, w)}},
    sort={uhatnkfw},
    description={Matrix coefficients defined in terms of an indeterminate parameter $w$ as $\hat{u}_{n,k}(f, w) := (w^k-1) \cdot u_{n,k}(f, w)$.},
    type={symbols},
    name={Matrix coefficients}
    }
\newglossaryentry{dsjn}{
    symbol={\ensuremath{\operatorname{ds}_j(f; n)}},
    sort={dsjn},
    description={Summands in the formula for the Dirichlet inverse of an arithmetic function defined on page \pageref{def_djn_Dn_func_defs}. The precise definition of this function is given by 
    \[
    \ds_j(f; n) = \begin{cases} 
     (-1)^{\delta_{n,1}} \widehat{f}(n), & \text{ if $j = 1$; } \\ 
     \sum\limits_{\substack{d|n \\ d>1}} \widehat{f}(d) \ds_{j-1}\left(f; \frac{n}{d}\right), & 
     \text{ if $j \geq 2$, } 
     \end{cases} 
    \]
    where the fixed function $\hat{f}$ is defined by glossary symbol 
    \glssymbol{fhatn}.},
    type={symbols},
    name={Inverse function component terms}
    }
\newglossaryentry{Dfn}{
    symbol={\ensuremath{D_f(n)}},
    sort={Dn},
    description={Function related to the Dirichlet inverse of a function $f$ defined on page \pageref{def_djn_Dn_func_defs}. 
    More precisely, this function is defined by the sum 
    $D_f(n) := \sum\limits_{j=1}^n \frac{\operatorname{ds}_{2j}(f; n)}{\hat{f}(1)^{2j+1}}$, where this definition involves the glossary symbols 
    \glssymbol{dsjn} and \glssymbol{fhatn}. 
    Lemma \ref{lemma_Dastfhat} 
    relates this function to the Dirichlet inverse of the function $\hat{f}(n)$.},
    type={symbols},
    name={Inverse function sum function}
    }
\newglossaryentry{fInvn}{
    symbol={\ensuremath{f^{-1}(n)}},
    sort={fInvn},
    description={The Dirichlet inverse of $f$ with respect to convolution defined recursively by $f^{-1}(n) = -\frac{1}{f(1)} \sum\limits_{\substack{d|n \\ d>1}} f(d) f^{-1}(n/d)$ provided that $f(1) \neq 0$.},
    type={symbols},
    name={Dirichlet inverse of $f$}
    }
\newglossaryentry{qPochInfty}{
    symbol={\ensuremath{(q; q)_{\infty}}},
    sort={qPochInfty},
    description={The infinite $q$-Pochhammer symbol defined as the product $(q; q)_{\infty} := \prod\limits_{n \geq 1} (1-q^n)$ for $|q| < 1$.},
    type={symbols},
    name={Infinite $q$-Pochhammer symbol}
    }
\newglossaryentry{Fk}{
    symbol={\ensuremath{F[k]}},
    sort={Fk},
    description={Discrete Fourier transform coefficients defined on page \pageref{page_DTFT_coefficient_defs}.},
    type={symbols},
    name={DFT coefficients}
    }
\newglossaryentry{skfgn}{
    symbol={\ensuremath{s_k(f, g; n)}},
    sort={skfgn},
    description={Shorthand for the periodic (modulo $k$) divisor sums defined on page \pageref{eqn_snk_akm_intro_Fourier_coeff_exps} and expanded by the functions listed in 
    \glssymbol{akfgn} of this glossary. The precise expansion and corresponding finite Fourier series expansion of this 
    function is given by $s_k(f, g; n) = \sum\limits_{d|(n, k)} f(d) g(k/d) = \sum\limits_{m=1}^{k} a_k(f, g; m) e^{2\pi\imath \cdot mn / k}$.},
    type={symbols},
    name={Shorthand for periodic divisor sums}
    }
\newglossaryentry{akfgn}{
    symbol={\ensuremath{a_k(f, g; n)}},
    sort={akfgn},
    description={Discrete Fourier coefficients of the periodic divisor sums $s_k(f, g; n)$ defined on page \pageref{eqn_snk_akm_intro_Fourier_coeff_exps} and as symbol \glssymbol{skfgn} in this glossary. The precise definition 
    of these sums is given by $a_k(f, g; n) = \sum\limits_{d|(k, n)} g(d) f(n/d) \frac{d}{k}$.},
    type={symbols},
    name={DFT coefficients}
    }
\newglossaryentry{fCvlDashVI}{
    symbol={\ensuremath{f \ast C_{-}(m)}},
    sort={fCvlDashVI},
    description={This notation indicates that the index over which we perform the Dirchlet convolution is given by the dash parameter, 
    $(f \ast C_{-}(m))(n) := \sum_{d|n} f(d) C_{n/d}(m)$.},
    type={symbols},
    name={Definition of convolution with a fixed parameter}
    }
\newglossaryentry{fCvlDashVII}{
    symbol={\ensuremath{f \ast C_k(-)}},
    sort={fCvlDashVII},
    description={This notation indicates that the index over which we perform the Dirchlet convolution is given by the dash parameter, 
    $(f \ast C_k(-))(n) := \sum_{d|n} f(d) C_k(n/d)$.},
    type={symbols},
    name={Definition of convolution with a fixed parameter}
    }
\newglossaryentry{akl}{
    symbol={\ensuremath{a_{k,\ell}}},
    sort={akl},
    description={Sequence of coefficients that are defined explicitly on page \pageref{def_NotationAndSpecialExpSums} in the discrete Fourier series expansion of the type II sums \glssymbol{Lfgkx}. These coefficients are implicitly defined by Definition \ref{def_NotationAndSpecialExpSums} by the sums 
    $L_{f,g,k}(n) = \sum\limits_{\ell=0}^{k-1} a_{k,\ell} \cdot e(\ell n/k)$, where \glssymbol{Expex} is the shorthand for the complex exponential terms in the exponential sums we define in the article.},
    type={symbols},
    name={Local coefficient definition}
    }
\newglossaryentry{Idkn}{
    symbol={\ensuremath{\operatorname{Id}_k(n)}},
    sort={Idkn},
    description={The power-scaled identity function, $\operatorname{Id}_k(n) := n^k$ for $n \geq 1$.},
    type={symbols},
    name={Dirchlet identity function for powers}
    }
\newglossaryentry{Floorx}{
    symbol={\ensuremath{\lfloor x \rfloor}},
    sort={Floorx},
    description={The floor function $\lfloor x \rfloor := x - \{x\}$ where $0 \leq \{x\} < 1$ denotes the fractional part of $x \in \mathbb{R}$.},
    type={symbols},
    name={Floor function}
    }
\newglossaryentry{Ceilingx}{
    symbol={\ensuremath{\lceil x \rceil}},
    sort={Ceilingx},
    description={The ceiling function $\lceil x \rceil := x + 1 - \{x\}$ where $0 \leq \{x\} < 1$ denotes the fractional part of $x \in \mathbb{R}$.},
    type={symbols},
    name={Ceiling function}
    }
\newglossaryentry{Pkwt}{
    symbol={\ensuremath{P_k(w, t)}},
    sort={Pkwt},
    description={Type of orthogonal polynomial function defined on page \pageref{subSection_orthog_conditions_integral_formulas}. It satisfies that $\sum\limits_{k=1}^n \hat{u}_{n,k}(f, w) P_k(w, t) = t^n$.},
    type={symbols},
    name={Local DTFT function}
    }
\newglossaryentry{omegat}{
    symbol={\ensuremath{\omega(t)}},
    sort={omegat},
    description={Orthogonal polynomial orthogonality weight function defined on page \pageref{eqn_gen_orthog_condition_stmt_v1}. This function is related to glossary symbol \glssymbol{Pkwt}.},
    type={symbols},
    name={Local DTFT transformation function}
    }
\newglossaryentry{Gj}{
    symbol={\ensuremath{G_j}},
    sort={Gj},
    description={Denotes the interleaved (or generalized) sequence of pentagonal numbers defined explictly by the formula $G_j := \frac{1}{2} \ceiling{\frac{j}{2}} \ceiling{\frac{3j+1}{2}}$. The sequence begins as 
    $\{G_j\}_{j \geq 0} = \{0, 1, 2, 5, 7, 12, 15, 22, 26, 35, 40, 51, \ldots\}$.},
    type={symbols},
    name={Interleaved pentagonal numbers}
    }
\newglossaryentry{Ckn}{
    symbol={\ensuremath{C_k(n)}},
    sort={Ckn},
    description={Sequence of nested $k$-convolutions of an arithmetic function $f$ with itself defined on page \pageref{page_eqn_Ckn_kCvls}. The precise definition of this sequence is given by 
    \[
     C_k(n) = \begin{cases} 
     \widehat{f}(n) - \widehat{f}(1)\varepsilon(n), & \text{ if $k = 1$; } \\ 
     \sum\limits_{d|n} \left(\widehat{f}(d) - \widehat{f}(1) \varepsilon(d)\right) 
     C_{k-1}(n/d), & \text{ if $k \geq 2$, } 
     \end{cases} 
     \] 
     where the symbol $\hat{f}(n)$ is defined in glossary entry \glssymbol{fhatn}.},
    type={symbols},
    name={Nested $k$-convolutions of $f$ with itself}
    }    
\glsaddall[types={symbols}]

\begin{document} 

\begin{abstract} 
We build on and generalize recent work on so-termed factorization theorems for Lambert series 
generating functions. These factorization theorems allow us to express formal generating 
functions for special sums as invertible matrix transformations involving partition functions.
In the Lambert series case, the generating functions at hand 
enumerate the divisor sum coefficients of $q^n$ as 
$\sum_{d|n} f(d)$ for some arithmetic function $f$. 
Our new factorization theorems provide analogs to these 
established expansions generating corresponding sums of the form 
$\sum_{d: (d,n)=1} f(d)$ (type I sums) and the Anderson-Apostol sums 
$\sum_{d|(m,n)} f(d) g(n/d)$ (type II sums) 
for any arithmetic functions $f$ and $g$. 
Our treatment of the type II sums includes a matrix-based factorization method 
relating the partition function $p(n)$ to arbitrary arithmetic functions $f$. 
Weconclude the last section of the article by directly expanding new 
formulas for an arithmetic function $g$ by the type II sums using discrete, and discrete 
time, Fourier transforms (DFT and DTFT) for functions over inputs of greatest common divisors. 
\end{abstract}

\maketitle

\section{Notation and conventions} 
\label{Section_Glossary_NotationConvs} 

To make a common source of references to definitions of key functions and sequences defined within the article, 
we provide a comprehensive list of commonly used notation and conventions. 
Where possible, we have included page references to local definitions of special 
sequences and functions where they are given in the text. We have organized the symbols list alphabetically by symbol. This listing of notation and conventions is a useful reference to accompany the article. 
It is provided in Appendix \ref{Appendix_Glossary_NotationConvs} starting on page 
\pageref{Appendix_Glossary_NotationConvs}. 

\section{Introduction} 
\label{Section_Intro}

\subsection{Motivation} 

The \emph{average order} of an arithmetic function $f$ is defined as the arithmetic mean of the summatory function, 
$F(x) := \sum_{n \leq x} f(n)$, as $F(x) / x$. We typically represent the average order of 
a function in the form of an asymptotic formula in the cases where the formula, $F(x) / x$, for the average order 
diverges as $x \rightarrow \infty$. For example, it is well known that the average order of $\Omega(n)$, which counts the 
number of prime factors of $n$ (counting multiplicity), is $\log\log n$, and that the average order of Euler's totient 
function, $\phi(n)$, is given by $\frac{6n}{\pi^2}$ \cite{HARDYWRIGHT}. 
We are motivated by considering the breakdown of the partial sums of an 
arithmetic function $f(d)$ whose average order we would like to estimate into sums over the 
pairwise disjoint sets of component indices $d \leq x$: 
\begin{align} 
\label{eqn_avg_order_sum_decomp_intro_v1} 
\sum_{d \leq x} f(d) & = \sum_{\substack{1 \leq d \leq x \\ (d,x)=1}} f(d) + \sum_{\substack{d|x \\ d>1}} f(d) + 
     \sum_{\substack{1 < d \leq x \\ 1 < (d,x) < x}} f(d). 
\end{align} 
In particular, in evaluating the partial sums of an arithmetic function $f(d)$ over all $d \leq x$, we wish 
to break the terms in these partial sums into three sets: those $d$ relatively prime to $x$, the 
$d$ dividing $x$ (for $d > 1$), and the 
somewhat less ``round'' set of indices $d$ which are neither relatively prime to $x$ nor proper divisors of $x$. 
Now if we let $f$ denote any arithmetic function, we define the remainder terms in our average order 
expansions from \eqref{eqn_avg_order_sum_decomp_intro_v1} as follows: 
\begin{equation} 
\label{eqn_Tfx_remainder_sum_terms_def} 
\widetilde{S}_{f}(x) = \sum_{d \leq x} f(d) - 
     \sum_{\substack{1 \leq d \leq x \\ (d,x)=1}} f(d) - \sum_{\substack{d|x \\ d > 1}} f(d). 
\end{equation} 
For instance, when $x = 24$ we have that 
\[
\widetilde{S}_f(24) = f(9) + f(10) + f(14) + f(16) + f(18) + f(20) + f(21) + f(22). 
\] 
We observe that the last divisor sum terms in \eqref{eqn_Tfx_remainder_sum_terms_def} correspond to the 
coefficients of powers of $q$ in the Lambert series generating function over $f$ in 
the following form considered in the next subsection: 
\[
\sum_{d|x} f(d) = [q^x] \left(\sum_{n \geq 1} \frac{f(n) q^n}{1-q^n}\right),\ |q| < 1. 
\] 
We can see that the average order sums on the left-hand-side of 
\eqref{eqn_avg_order_sum_decomp_intro_v1} correspond to the hybrid of 
divisor and relatively prime divisor sums of the form 
\[
\sum_{d \leq x} f(d) = \sum_{m|x} 
     \sum_{\substack{k=1 \\ \left(k, \frac{x}{m}\right) = 1}}^{\frac{x}{m}} f(km) = 
     \sum_{m|x} \sum_{\substack{k=1 \\ (k, m) = 1}}^{m} f\left(\frac{kx}{m}\right). 
\] 
In this article, we study and prove new results relating both variants of the sums expanding the 
right-hand-side of the previous equation to restricted partitions and special 
partition functions. Namely, a combination of the results we prove in 
Section \ref{Section_Sum_TypeI} and the Lambert series factorization theorem results 
summarized in the next subsection allow us to write 
\[
\sum_{n \leq x} f(d) = \sum_{d|x+1} s_{x,d} \left[ 
     \sum_{j=0}^d \sum_{k=1}^{j-1} \sum_{i=0}^j (-1)^{\ceiling{\frac{i}{2}}} p(d-j) 
     \chi_{1,k}(j-k-G_i) \Iverson{j-k-G_i \geq 1} \cdot 
     f\left(\frac{(x+1)k}{d}\right) 
     \right], 
\] 
where $\chi_{1,k}(n)$ is the principal Dirichlet character modulo $k$, 
$\Iverson{n = k} = \delta_{n,k}$ denotes Iverson's convention, 
$G_j := \frac{1}{2} \ceiling{j/2} \ceiling{(3j+1)/2}$ denotes the 
sequence of interleaved, or generalized 
pentagonal numbers, the triangular sequence 
\label{page_snk_LSeriesFactMatrixCoeffs_ref} 
$s_{n,k} := [q^n] (q; q)_{\infty} \frac{q^k}{1-q^k}$ corresponds to the 
difference of restricted partition functions discussed in the next subsection, 
and $p(n)$ is the classical partition function. 
The analysis of the asymptotic properties of these sums is a central topic in the 
study of the behavior of arithmetic functions, analytic number theory, and in 
applications such as algorithmic analysis. 
Our new results connect variants of such sums over multiplicative functions with the 
distinctly additive flavor of the theory of partitions. 

\subsection{Variations on recent work} 

There is a fairly complete and extensive set of expansions providing identities related to 
these Lambert series generating functions 
and their matrix factorizations in the form of so-termed ``Lambert series factorization theorems'' studied by 
Merca and Schmidt over 2017--2018 \citep{AA,MERCA-LSFACTTHM,MERCA-SCHMIDT-LSFACTTHM}. 
These results provide factorizations for a Lambert series generating function 
over the arbitrary arithmetic function $f$ expanded in the form of 
\begin{align} 
\label{eqn_LambertSeriesFactThm_v1} 
L_f(q) := 
\sum_{n \geq 1} \frac{f(n) q^n}{1-q^n} & = \frac{1}{(q; q)_{\infty}} \sum_{n \geq 1} \left( 
     \sum_{k=1}^n s_{n,k} f(k)\right) q^n, 
\end{align} 
where $s_{n,k} = s_o(n, k) - s_e(n, k)$ is independent of $f$ and is defined as the difference of the 
functions $s_{o/e}(n, k)$ which respectively denote the number of 
$k$'s in all partitions of $n$ 
into an odd (even) number of distinct parts. 
These so-termed factorization theorems, which effectively provide a matrix-based expansion of an 
ordinary generating function for the divisor sums of the type enumerated by Lambert series 
expansions, connect the additive theory of partitions to the more multiplicative constructions 
of power series generating functions found in other branches on number theory. 
In these cases, it appears that it is most natural, in some sense, to expand these sums via the 
factorizations defined in \eqref{eqn_LambertSeriesFactThm_v1} 
since the matrix entries (and their inverses) are also partition-related. 
It then leads us to the question of what other natural, or even canonical, analogous expansions 
can be formed for other more general variants of the above divisor sums. 

More generally, we can form analogous matrix-based factorizations of the 
generating functions of the sequences of special sums in 
\eqref{eqn_GenMatrixBasedSumsEnum_SetsAn} provided that these 
transformations are invertible. That is, we can express generating functions for the 
sums $s_n(f, \mathcal{A}) := \sum_{k \in \mathcal{A}_n} f(k)$ where we take 
$\mathcal{A}_n \subseteq [1,n] \cap \mathbb{Z}$ for all $n \geq 1$: 
\begin{equation} 
\label{eqn_GenMatrixBasedSumsEnum_SetsAn} 
\sum_{n \geq 1} \left(\sum_{\substack{k \in \mathcal{A}_n \\ 
     \mathcal{A}_n \subseteq [1,n]}} f(k)\right) q^n = 
     \frac{1}{(q; q)_{\infty}} 
     \sum_{n \geq 1} \left(\sum_{k=1}^n v_{n,k}(\mathcal{A}) f(k)\right) q^n, 
\end{equation} 
The sums in \eqref{eqn_sum_vars_TL_defs_intro} below are referred to as 
\emph{type I} and \emph{type II sums}, respectively, 
in the next subsections. These sums are given by the 
special cases of \eqref{eqn_GenMatrixBasedSumsEnum_SetsAn} where 
when $\mathcal{A}_{1,n} := \{d : 1 \leq d \leq n, (d, n) = 1\}$ and 
$\mathcal{A}_{2,n} := \{d : 1 \leq d \leq n, d | (k, n)\}$ for some $1 \leq k \leq n$, 
respectively. 
\begin{align} 
\label{eqn_sum_vars_TL_defs_intro} 
T_{f}(x) & = \sum_{\substack{d=1 \\ (d,x)=1}}^x f(d) \\ 
\notag 
L_{f,g,k}(x) & = \sum_{d|(k,x)} f(d) g\left(\frac{x}{d}\right) 
\end{align} 
We define the following preliminary constructions for the factorizations of the Lambert-like 
series whose respective expansions involve the sums in \eqref{eqn_sum_vars_TL_defs_intro}. Notice that the 
invertible matrix coefficients, $t_{n,k}$ and $u_{n,k}(f, w)$, in each expansion are defined such 
that the subequations in \eqref{eqn_intro_two_sum_fact_types_def} are correct. 
We will identify precise formulas for these invertible matrices as primary results in the article: 
\begin{subequations} 
\label{eqn_intro_two_sum_fact_types_def} 
\begin{align} 
\label{eqn_intro_two_sum_fact_types_def_v1}
T_{f}(x) & = [q^x]\left(\frac{1}{(q; q)_{\infty}} \sum_{n \geq 2} \sum_{k=1}^n t_{n,k} f(k) \cdot q^n + 
     f(1) \cdot q\right) \\ 
\label{eqn_intro_two_sum_fact_types_def_v2}
g(x) & = [q^x]\left(\frac{1}{(q; q)_{\infty}} \sum_{n \geq 2} \sum_{k=1}^n u_{n,k}(f, w) \left[ 
     \sum_{m=1}^k L_{f,g,m}(k) w^m\right] \cdot q^n\right),\ w \in \mathbb{C}. 
\end{align} 
The sequences $t_{n,k}$ and $u_{n,k}(f, w)$ are lower triangular and 
invertible for suitable choices of the indeterminate parameter $w$. 
For a fixed $N \geq 1$, we can truncate these sequences after $N$ rows 
and form the $N \times N$ matrices whose entries are 
$t_{n,k}$ (respectively, $u_{n,k}(f, w)$) for $1 \leq n,k \leq N$. 
The corresponding inverse matrices have terms denoted by 
$t_{n,k}^{(-1)}$ (and $u_{n,k}^{(-1)}(f, w)$) respectively). 
That is to say, for $n \geq 2$, these inverse matrices satisfy 
\begin{align} 
\label{eqn_intro_two_sum_fact_types_def_v3_invs}
f(n) & = \sum_{k=1}^n t_{n,k}^{(-1)} \cdot [q^k]\left((q; q)_{\infty} \times \sum_{n \geq 1} 
     T_f(n) q^n\right) \\ 
\label{eqn_intro_two_sum_fact_types_def_v4_invs}
\sum_{m=1}^n L_{f,g,m}(k) w^m & = \sum_{k=1}^n u_{n,k}^{(-1)}(f, w) \cdot [q^k] \left((q; q)_{\infty} 
     \times \sum_{n \geq 1} g(n) q^n\right),\ w \in \mathbb{C}.
\end{align} 
Explicit representations for these inverse matrix sequences are proved in the article. 
\end{subequations} 
We focus on the special expansions of each factorization type in 
Section \ref{Section_Sum_TypeI} and Section \ref{Section_Sum_TypeII}, respectively, 
though we note that other related variants of these expansions are possible. 

\begin{theorem}[Exact Formulas for the Factorization Matrix Sequences] 
\label{theorem_snk_snkinv_seq_formulas_v1} 
The lower triangular sequence $t_{n,k}$ is defined by the first expansion 
in \eqref{eqn_intro_two_sum_fact_types_def_v1}. 
The corresponding inverse matrix coefficients are denoted by $t_{n,k}^{(-1)}$. 
For integers $n, k \geq 1$, the two lower triangular 
factorization sequences defining the expansion of 
\eqref{eqn_intro_two_sum_fact_types_def_v1} satisfy exact formulas given by 
\begin{align*} 
\tag{i} 
t_{n,k} & = \sum_{j=0}^{n} (-1)^{\lceil j/2 \rceil} \chi_{1,k}(n+1-G_j) \Iverson{n-G_j \geq 1} \\ 
\tag{ii} 
t_{n,k}^{(-1)} & = \sum_{d=1}^n p(d-k) \mu_{n,d}, 
\end{align*} 
where we define the sequence of interleaved pentagonal numbers $G_j$ 
as in the introduction, and the sequence $\mu_{n,k}$ as in 
Proposition \ref{prof_inversion_formula}. 
\end{theorem} 

The function $\chi_{1,k}(n)$ defined in the 
glossary section starting on page \pageref{Section_Glossary_NotationConvs} 
refers to the principal Dirichlet character modulo $k$ for some $k \geq 1$. 

\begin{prop}[Formulas for the Inverse Matrix Sequences of $u_{n,k}(f, w)$] 
\label{prop_unk_inverse_matrix} 
For all $n \geq 1$ and $1 \leq k \leq n$, any fixed arithmetic function $f$, 
and $w \in \mathbb{C}$, we have that 
\[
u_{n,k}^{(-1)}(f, w) = \sum_{m=1}^n \left(\sum_{d|(m,n)} f(d) p(n/d-k)\right) w^m. 
\]
\end{prop} 

Another formulation of the expression for the inverse sequence in the 
previous proposition is proved in 
Proposition \ref{prop_another_matrix_formula} using constructions 
we define in the subsections leading up to that result. 

\subsection{A summary of applications of our new results} 

\subsubsection{Forms of the type I sums} 

The identities and theorems we prove for the general sum case defined by
\eqref{eqn_sum_vars_TL_defs_intro} in Section \ref{Section_Sum_TypeI} can be useful in constructing new 
undiscovered identities for well-known functions. These
expansions which are phrased in terms of our new matrix factorization sequences are not well explored
yet. So these expansions may yield additional information on the properties of the well-known function cases, 
which are examples of the type I sums. We give a few notable examples of summation identities which express 
classical functions and combinatorial objects in new ways below. These ways are to illustrate the style of an 
application of our new formulas. The concrete examples we cite below also serve to motivate the methods behind 
our new matrix factorizations for the special type I sums. We will set out to prove these identities in more 
generality in later sections of this article.

We obtain the following identities for 
Euler's totient function based on our new constructions: 
\begin{align*} 
\phi(n) & = \sum_{j=0}^n \sum_{k=1}^{j-1} \sum_{i=0}^{j} 
     p(n-j) (-1)^{\lceil i/2 \rceil} \chi_{1,k}(j-k-G_i) \Iverson{j-k-G_i \geq 1} + 
     \Iverson{n=1} \\ 
\phi(n) & = \sum_{\substack{d=1 \\ (d,n)=1}}^n 
     \left(\sum_{k=1}^{d+1} \sum_{i=1}^d \sum_{j=0}^k 
     p(i+1-k) (-1)^{\lceil j/2 \rceil} \phi(k-G_j) \mu_{d,i} \Iverson{k-G_j \geq 1}\right). 
\end{align*} 
To give another related example that applies to classical multiplicative functions, 
recall that we have a known representation for the M\"obius function given as an 
exponential sum in terms of powers of the 
$n^{th}$ primitive roots of unity of the form \citep[\S 16.6]{HARDYWRIGHT} 
\[ 
\mu(n) = \sum_{\substack{d=1 \\ (d,n)=1}}^n \exp\left(2\pi\imath \frac{d}{n}\right). 
\] 
The \emph{Mertens function}, $M(x)$, is defined as the summatory function over the 
M\"obius function $\mu(n)$ for all $n \leq x$. Using the definition of the M\"obius function as one of our 
type I sums defined above, we have new expansions for the Mertens function given by 
(\cf Corollary \ref{cor_Mertens_function_v2}) 
\[
M(x) = \sum_{1 \leq k < j \leq n \leq x} \left(\sum_{i=0}^j p(n-j) (-1)^{\lceil i/2 \rceil} 
     \chi_{1,k}(j-k-G_i) \Iverson{j-k-G_i \geq 1} e^{2\pi\imath k / n}\right). 
\] 
Finally, we can form another related polynomial sum of the type indicated above 
when we consider that the logarithm of the \emph{cyclotomic polynomials} leads to the sums 
\begin{align*} 
\log \Phi_n(z) & = \sum_{\substack{1 \leq k \leq n \\ (k, n) = 1}} 
     \log\left(z-e^{2\pi\imath k / n}\right) \\ 
     & = 
     \sum_{1 \leq k < j \leq n} \left(\sum_{i=0}^j p(n-j) (-1)^{\lceil i/2 \rceil} 
     \chi_{1,k}(j-k-G_i) \Iverson{j-k-G_i \geq 1} 
     \log\left(z-e^{2\pi\imath k / n}\right)
     \right). 
\end{align*} 

\subsubsection{Forms of the type II sums} 

The sums $L_{f,g,k}(n)$ are sometimes refered to as \emph{Anderson-Apostol sums} named 
after the authors who first defined them 
(\cf \cite[\S 8.3]{APOSTOL-ANUMT} \cite{APOSTOL-APGRS}). 
Other variants and generalizations of these sums are studied in the 
references \cite{IKEDA,KIUCHI-AVGS}. 
There are many number theoretic applications of the periodic sums factorized in this form. 
For example, the famous expansion of Ramanujan's sum $c_q(n)$ is expressed as the 
following right-hand-side divisor sum \cite[\S IX]{RAMANUJAN}: 
\[
c_q(n) = \sum_{\substack{d=1 \\ (d, n)=1}}^n 
     e^{2\pi\imath dn / q} = \sum_{d|(q,n)} d \cdot \mu(q/d). 
\] 
The applications of our new results to Ramanujan's sum include the 
expansions 
\begin{align*} 
c_n(x) & = [w^{x}]\left(\sum_{k=1}^n u_{n,k}^{(-1)}(\mu, w) \sum_{j \geq 0} 
     (-1)^{\lceil j/2 \rceil} \mu(k-G_j)\right) \\ 
     & = 
     \sum_{k=1}^n \left(\sum_{d|(n,x)} d \cdot p(n/d-k)\right) \sum_{j \geq 0} 
     (-1)^{\lceil j/2 \rceil} \mu(k-G_j), 
\end{align*} 
where the inverse matrices $u_{n,k}^{(-1)}(\mu, w)$ 
are expanded according to Proposition \ref{prop_unk_inverse_matrix}. 
We then immediately have the following new results for the next special expansions of the 
generalized sum-of-divisors functions when $\Re(s) > 0$: 
\begin{align*} 
\sigma_s(n) & = n^s \zeta(s+1) \times \sum_{i=1}^{\infty} \sum_{k=1}^i \left(\sum_{d|(n,i)} 
     d \cdot p(i/d-k)\right) \sum_{j \geq 0} \frac{(-1)^{\lceil j/2 \rceil} \mu(k-G_j)}{i^{s+1}} \\ 
     & = 
     n^s \zeta(s+1) \times \sum_{i=1}^n \sum_{k \geq 0} \frac{\mu\left(\frac{nk+i}{(n,i)}\right)}{(nk+i)^{s+1}}. 
\end{align*} 
Section \ref{subSection_DFT_exps} expands the left-hand-side function 
$g(x)$ in \eqref{eqn_intro_two_sum_fact_types_def_v2} by considering a new 
indirect method involving the type II sums $L_{f,g,k}(n)$. The expansions we derive in that section 
employ results for discrete Fourier 
transforms of functions of the greatest common divisor studied in 
\cite{kamp-gcd-transform,SCHRAMM}. 
This method allows us to study the factorization forms in 
\eqref{eqn_intro_two_sum_fact_types_def_v2} where we effectively bypass the complicated forms of the 
ordinary matrix coefficients $u_{n,k}(f, w)$. Results enumerating the ordinary matrices with coefficients 
given by $u_{n,k}(f, w)$ are treated in 
Corollary \ref{cor_unkfw_ord_matrix_formula_v1} of 
Section \ref{subSection_typeII_inv_matrices}. 
The discrete Fourier series methods we use to prove our theorems in these sections 
lead to the next key result proved in Theorem \ref{theorem_main_lfgmn_exp_v1}
which states that 
\[
\sum_{d|k} \sum_{r=0}^{k-1} d \cdot L_{f,g,r}(k) \e{-\frac{rd}{k}} \mu(k/d) = 
  \sum_{d|k} \phi(d) f(d) (k/d)^2 g(k/d),
\] 
where $\e{x} = \exp(2\pi\imath \cdot x)$ is standard notation for the 
complex exponential function. 

\subsection{Significance of our new results} 

Our new results provide generating function expansions for the type I and type II 
sums in the form of matrix-based factorization theorems. 
The matrix products involved in expressing the coefficients of these generating 
functions for arbitrary arithmetic functions $f$ and $g$ are closely related to the 
partition function $p(n)$. The known Lambert series factorization theorems 
proved in the references and which are summarized in 
the subsections on variants above 
demonstrate the flavor of the matrix-based expansions of these forms for 
ordinary divisor sums of the form $\sum_{d|n} f(d)$. 
Our extensions of these factorization theorem approached in the context of the 
new forms of the type I and type II sums similarly relate special arithmetic functions 
in number theory to partition functions and more additive branches of number theory. 
The last results proved in 
Section \ref{subSection_DFT_exps} are expanded in the spirit of these 
matrix factorization constructions using discrete Fourier transforms of functions 
(and sums of functions) evaluated at greatest common divisors. We pay special 
attention to illustrating our new results with many relevant examples and 
new identities expanding famous special number theoretic functions throughout the 
article. 

\section{Factorization theorems for sums of the first type} 
\label{Section_Sum_TypeI}

\subsection{Inversion relations} 

We begin our exploration here by expanding an inversion formula which is analogous to 
M\"obius inversion for ordinary divisor sums. 
We prove the following result which is the analog to the sequence 
inversion relation provided by the M\"obius transform in the context of our sums over the 
integers relatively prime to $n$ \cite[\cf \S 2, \S 3]{RIORDAN-COMBIDENTS}. 

\begin{prop}[Inversion Formula] 
\label{prof_inversion_formula}
For all $n \geq 2$, there is a unique lower triangular sequence, denoted by $\mu_{n,k}$, 
which satisfies the lower triangular inversion relation (i.e., so that $\mu_{n,d} = 0$ whenever $n < d$) 
\begin{subequations}
\begin{equation}
\label{eqn_propInvFormula_v1}
g(n) = \sum_{\substack{d=1 \\ (d, n)=1}}^n 
     f(d) \quad\iff\quad f(n) = \sum_{d=1}^n g(d+1) \mu_{n,d}. 
\end{equation}
Moreover, if we form the matrix $(\mu_{i,j} \Iverson{j \leq i})_{1 \leq i,j \leq n}$ for any $n \geq 2$, we have that the 
inverse sequence satisfies 
\begin{equation} 
\label{eqn_propInvFormula_v2}
\mu_{n,k}^{(-1)} = \Iverson{(n+1, k) = 1} \Iverson{k \leq n}. 
\end{equation} 
\end{subequations} 
\end{prop} 
\begin{proof} 
Consider the $(n-1) \times (n-1)$ matrix 
\begin{equation}
\label{eqn_InvRel_proof_matrix_def_v1} 
\left(\Iverson{(i,j-1)=1\mathrm{\ and \ } j \leq i}\right)_{1 \leq i,j < n},
\end{equation} 
which effectively corresponds to the formula on the left-hand-side of \eqref{eqn_propInvFormula_v1} by 
applying the matrix to the vector of $[f(1)\ f(2)\ \cdots f(n)]^{T}$ and extracting the 
$(n+1)^{th}$ column of the matrix formed by extracting the $\{0,1\}$-valued coefficients of $f(d)$. Since $\gcd(i, j-1)=1$ for all $i = j$ with 
$i,j \geq 1$, we see that the 
matrix \eqref{eqn_InvRel_proof_matrix_def_v1} 
is lower triangular with ones on its diagonal. Thus the matrix is non-singular and its 
unique inverse, which we denote by $(\mu_{i,j})_{1 \leq i,j < n}$, leads to the sum on the 
right-hand-side of the sum in \eqref{eqn_propInvFormula_v1} when we shift $n \mapsto n+1$. 
The second equation stated in \eqref{eqn_propInvFormula_v2} 
restates the form of the first matrix of $\mu_{i,j}$ 
as on the right-hand-side of \eqref{eqn_propInvFormula_v1}. 
\end{proof}

\begin{figure}[ht!]

\begin{minipage}{\linewidth} 
\begin{center} 
\small
\begin{equation*} 
\boxed{ 
\begin{array}{ccccccccccccccccc}
 1 & 0 & 0 & 0 & 0 & 0 & 0 & 0 & 0 & 0 & 0 & 0 & 0 & 0 & 0 & 0 & 0 \\
 -1 & 1 & 0 & 0 & 0 & 0 & 0 & 0 & 0 & 0 & 0 & 0 & 0 & 0 & 0 & 0 & 0 \\
 -1 & 0 & 1 & 0 & 0 & 0 & 0 & 0 & 0 & 0 & 0 & 0 & 0 & 0 & 0 & 0 & 0 \\
 1 & -1 & -1 & 1 & 0 & 0 & 0 & 0 & 0 & 0 & 0 & 0 & 0 & 0 & 0 & 0 & 0 \\
 -1 & 0 & 0 & 0 & 1 & 0 & 0 & 0 & 0 & 0 & 0 & 0 & 0 & 0 & 0 & 0 & 0 \\
 1 & 0 & 0 & -1 & -1 & 1 & 0 & 0 & 0 & 0 & 0 & 0 & 0 & 0 & 0 & 0 & 0 \\
 1 & 0 & -1 & 0 & -1 & 0 & 1 & 0 & 0 & 0 & 0 & 0 & 0 & 0 & 0 & 0 & 0 \\
 -1 & 0 & 2 & -1 & 0 & 0 & -1 & 1 & 0 & 0 & 0 & 0 & 0 & 0 & 0 & 0 & 0 \\
 -1 & 0 & 0 & 0 & 1 & 0 & -1 & 0 & 1 & 0 & 0 & 0 & 0 & 0 & 0 & 0 & 0 \\
 1 & 0 & -1 & 1 & 0 & -1 & 1 & -1 & -1 & 1 & 0 & 0 & 0 & 0 & 0 & 0 & 0 \\
 -1 & 0 & 1 & 0 & 0 & 0 & -1 & 0 & 0 & 0 & 1 & 0 & 0 & 0 & 0 & 0 & 0 \\
 1 & 0 & -1 & 0 & 0 & 0 & 1 & 0 & 0 & -1 & -1 & 1 & 0 & 0 & 0 & 0 & 0 \\
 3 & 0 & -2 & 0 & -2 & 0 & 2 & 0 & -1 & 0 & -1 & 0 & 1 & 0 & 0 & 0 & 0 \\
 -3 & 0 & 1 & 0 & 3 & 0 & -1 & -1 & 1 & 0 & 0 & 0 & -1 & 1 & 0 & 0 & 0 \\
 -1 & 0 & 1 & 0 & 1 & 0 & -1 & 0 & 0 & 0 & 0 & 0 & -1 & 0 & 1 & 0 & 0 \\
 1 & 0 & 0 & 0 & -2 & 0 & 0 & 1 & 0 & 0 & 1 & -1 & 1 & -1 & -1 & 1 & 0 \\
 -3 & 0 & 2 & 0 & 2 & 0 & -2 & 0 & 1 & 0 & 0 & 0 & -1 & 0 & 0 & 0 & 1 \\
\end{array}
}
\end{equation*}
\end{center} 
\subcaption*{$\mu_{n,k}$ for $1 \leq n,k < 18$} 
\end{minipage} 

\caption{Inversion formula coefficient sequences } 
\label{figure_munk_and_inv_seqs} 

\end{figure} 

\begin{remark} 
Figure \ref{figure_munk_and_inv_seqs} provides a listing of the relevant analogs to the role of the 
M\"obius function in a M\"obius inversion transform of the ordinary divisor sum over 
an arithmetic function. 
We do not know of a comparatively simple closed-form function for the sequence of 
$\mu_{n,k}$ \cite[\cf A096433]{OEIS}. However, we readily see by construction that the 
sequence and its inverse satisfy 
\begin{align*} 
\sum_{\substack{d=1 \\ (d,n)=1}}^n \mu_{d,k} & = 0 \\ 
\sum_{\substack{d=1 \\ (d,n)=1}}^n \mu_{d,k}^{(-1)} & = \phi(n), 
\end{align*} 
where $\phi(n)$ is Euler's totient function. 
The first columns of the $\mu_{n,1}$ appear in the integer sequences 
database as the entry \cite[A096433]{OEIS}. 
\end{remark} 

\subsection{Exact formulas for the factorization matrices} 

The next result is key to proving the exact formulas for the matrix sequences, 
$t_{n,k}$ and $t_{n,k}^{(-1)}$, and their expansions by the partition functions defined 
in the introduction. We prove the following result
first as a lemma which we will use in the proof of 
Theorem \ref{theorem_snk_snkinv_seq_formulas_v1} given below. 
The first several rows of the matrix sequence $t_{n,k}$ and its inverse 
implicit to the factorization theorem in \eqref{eqn_intro_two_sum_fact_types_def} are tabulated in 
Figure \ref{figure_lseriesfact_s3nk_tables} for intuition on the 
formulas we prove in the next proposition and following theorem. 

\begin{lemma}[A Convolution Identity for Relatively Prime Integers] 
\label{lemma_cvl_ident_rpints_indicatorfn} 
For all natural numbers $n \geq 2$ and $k \geq 1$ with $k \leq n$, we have the 
following expression for the principal Dirichlet character modulo $k$: 
\[ 
\sum_{j=1}^n t_{j,k} p(n-j) = \chi_{1,k}(n). 
\] 
Equivalently, we have that 
\begin{align} 
\label{eqn_Proof_cvl_ident_rpints_indicatorfn_EQ1}
t_{n,k} & = \sum_{i=0}^n (-1)^{\lceil i/2 \rceil} \chi_{1,k}(n-G_i) \Iverson{n-G_i \geq k+1} \\ 
\notag
     & = 
     \sum_{b = \pm 1} \left[\sum_{i=0}^{\floor{\frac{\sqrt{24(n-k-1)+1}-b}{6}}} (-1)^{\ceiling{i/2}} \chi_{1,k}\left(n - \frac{i(3i-b)}{2}\right)
     \right].  
\end{align} 
\end{lemma} 
\begin{proof} 
We begin by noticing that the right-hand-side expression in the statement of the lemma 
is equal to $\mu_{n,k}^{(-1)}$ by the construction of the sequence in 
Proposition \ref{prof_inversion_formula}. 
Next, we see that the factorization in 
\eqref{eqn_intro_two_sum_fact_types_def_v1} is equivalent to the 
expansion 
\begin{align}
\label{eqn_tnk_formula_proof_formula_v2}  
\sum_{d=1}^{n-1} f(d) \mu_{n,d}^{(-1)} & = 
     \sum_{j=1}^n \sum_{k=1}^{j} p(n-j) t_{j,k} \cdot f(k). 
\end{align} 
Since $\mu_{n,k}^{(-1)} = \Iverson{(n+1, k) = 1}$, we may take the coefficients of $f(k)$ on 
each side of \eqref{eqn_tnk_formula_proof_formula_v2} 
for each $1 \leq k < n$ to establish the result we have claimed in this lemma. 
The equivalent statement of the result follows by a generating function argument 
applied to the product that generates the left-hand-side Cauchy product in 
\eqref{eqn_Proof_cvl_ident_rpints_indicatorfn_EQ1}. 
\end{proof} 

\begin{figure}[ht!]

\begin{minipage}{\linewidth} 
\begin{center} 
\small
\begin{equation*} 
\boxed{ 
\begin{array}{cccccccccccccc}
1 & \\
0 & 1 & \\
-1 & -1 & 1 & \\
-1 & 0 & 0 & 1 & \\
-1 & -1 & -2 & -1 & 1 & \\
0 & 0 & 0 & 0 & 0 & 1 & \\
0 & 0 & 0 & -1 & -1 & -1 & 1 & \\
1 & 0 & -1 & 0 & -1 & -1 & 0 & 1 & \\
1 & 1 & 1 & 0 & -2 & 0 & -1 & -1 & 1 & \\
1 & 0 & 1 & 0 & 1 & 1 & -1 & 0 & 0 & 1 & \\
1 & 1 & 0 & 1 & 1 & 0 & -1 & -1 & -2 & -1 & 1 & \\
1 & 0 & 1 & 0 & 1 & 0 & 0 & 0 & 0 & 0 & 0 & 1 & \\
0 & 1 & 1 & 1 & 1 & 0 & -1 & 0 & 0 & -1 & -1 & -1 & 1 & \\
0 & -1 & 0 & 0 & -1 & -1 & 2 & 0 & -1 & -1 & -1 & -1 & 0 & 1 \\
\end{array}
}
\end{equation*}
\end{center} 
\subcaption*{\rm{(i)} $t_{n,k}$} 
\end{minipage} 

\begin{minipage}{\linewidth} 
\begin{center} 
\small
\begin{equation*} 
\boxed{ 
\begin{array}{ccccccccccccc}
 1 & \\
 0 & 1 & \\
 1 & 1 & 1 & \\
 1 & 0 & 0 & 1 & \\
 4 & 3 & 2 & 1 & 1 & \\
 0 & 0 & 0 & 0 & 0 & 1 & \\
 5 & 3 & 2 & 2 & 1 & 1 & 1 & \\
 4 & 4 & 3 & 1 & 1 & 1 & 0 & 1 & \\
 15 & 11 & 8 & 5 & 4 & 2 & 1 & 1 & 1 & \\
 -1 & -1 & -1 & 1 & 0 & 0 & 1 & 0 & 0 & 1 & \\
 32 & 24 & 18 & 12 & 9 & 6 & 4 & 3 & 2 & 1 & 1 & \\
 -6 & -4 & -3 & -1 & -1 & 0 & 0 & 0 & 0 & 0 & 0 & 1 & \\
 24 & 17 & 13 & 12 & 8 & 7 & 6 & 3 & 2 & 2 & 1 & 1 & 1 \\
\end{array}
}
\end{equation*}
\end{center} 
\subcaption*{\rm{(ii)} $t_{n,k}^{(-1)}$} 
\end{minipage}

\caption{The factorization matrices, $t_{n,k}$ and $t_{n,k}^{(-1)}$, for 
         $1 \leq n,k < 14$} 
\label{figure_lseriesfact_s3nk_tables} 

\end{figure} 

\begin{proof}[Theorem \ref{theorem_snk_snkinv_seq_formulas_v1}: Proof of (i)]  
It is plain to see by the considerations in our construction of the factorization theorem that 
both matrix sequences are lower triangular. Thus, we need only consider the cases where 
$n \leq k$. By a convolution of generating functions, the identity in 
Lemma \ref{lemma_cvl_ident_rpints_indicatorfn} shows that 
\[ 
t_{n,k} = \sum_{j=k}^n [q^{n-j}] (q; q)_{\infty} \cdot \Iverson{(j+1, k) = 1}. 
\] 
Then shifting the index of summation in the previous equation implies (i). 
\end{proof} 
\begin{proof}[Theorem \ref{theorem_snk_snkinv_seq_formulas_v1}: Proof of (ii)] 
To prove (ii), we consider the factorization theorem when $f(n) := t_{n,r}^{(-1)}$ for some 
fixed $r \geq 1$. We then expand \eqref{eqn_intro_two_sum_fact_types_def_v1} as 
\begin{align*} 
\sum_{\substack{d=1 \\ (d,n)=1}}^n 
     t_{d,r}^{(-1)} & = [q^n] \frac{1}{(q; q)_{\infty}} \sum_{n \geq 1} \sum_{k=1}^{n-1} 
     t_{n,k} \cdot t_{k,r}^{(-1)} \cdot q^n \\ 
     & = 
     \sum_{j=1}^n p(n-j) \times \sum_{k=1}^{j-1} t_{j,k} t_{k,r}^{(-1)} \\ 
     & = 
     \sum_{j=1}^n p(n-j) \Iverson{r=j-1} \\ 
     & = 
     p(n-1-r). 
\end{align*} 
Hence we may perform the inversion by Proposition \ref{prof_inversion_formula} to the 
left-hand-side sum in the previous equations to obtain our stated result. 
\end{proof} 

\begin{remark}[Relations to the Lambert Series Factorization Theorems]
We notice that by inclusion-exclusion applied to the right-hand-side of 
\eqref{eqn_intro_two_sum_fact_types_def_v1}, 
we may write our matrices $t_{n,k}$ in terms of the triangular 
sequence expanded as differences of restricted partitions in 
\eqref{eqn_LambertSeriesFactThm_v1}. For example, when $k := 12$ we see that 
\[
\sum_{n \geq 12} \Iverson{(n, 12) = 1} q^n = \frac{q^{12}}{1-q} - \frac{q^{12}}{1-q^2} - 
     \frac{q^{12}}{1-q^3} + \frac{q^{12}}{1-q^6}. 
\]
In general, when $k > 1$ we can expand 
\[
\sum_{n \geq k} \Iverson{(n, k) = 1} q^n = \sum_{d|k} \frac{q^k \mu(d)}{1-q^d}. 
\]
Thus we can relate the triangles $t_{n,k}$ in this article to the 
$s_{n,k} = [q^n] (q; q)_{\infty} q^k / (1-q^k)$ for $n \geq k \geq 1$ employed in the 
expansions from the references as follows: 
\[
t_{n,k} = \begin{cases} 
     s_{n,k}, & k = 1; \\ 
     \sum\limits_{d|k} s_{n+1-k+d,d} \cdot \mu(d), & k > 1. 
     \end{cases} 
\]
\end{remark} 

\subsection{Completing the proofs of the main applications} 

We remark that as in the Lambert series factorization results from the references \cite{MERCA-LSFACTTHM}, 
we have three primary types of expansion identities that we will consider for 
any fixed choice of the arithmetic function $f$ in the forms of 
\begin{subequations} 
\label{eqn_factthm_exp_idents_v1} 
\begin{align} 
\sum_{\substack{d=1 \\ (d,n)=1}}^n 
     f(d) & = \sum_{j=1}^n \sum_{k=1}^{j-1} p(n-j) t_{j-1,k} f(k) + 
     f(1) \Iverson{n = 1} \\ 
\sum_{k=1}^{n-1} t_{n-1,k} f(k) & = \sum_{j=1}^n \sum_{\substack{d=1 \\ (d,j)=1}}^j 
     [q^{n-j}] (q; q)_{\infty} \cdot f(d) - [q^{n-1}](q; q)_{\infty} \cdot f(1), 
\end{align} 
and the 
corresponding inverted formula providing that 
\begin{align} 
\label{eqn_factthm_exp_idents_v1_eq.III} 
f(n) & = \sum_{k=1}^{n} t_{n,k}^{(-1)}\left( 
     \sum_{\substack{j \geq 0 \\ k+1-G_j > 0}} (-1)^{\ceiling{\frac{j}{2}}} 
     T_f(k+1-G_j) - [q^k] (q; q)_{\infty} \cdot f(1)\right). 
\end{align} 
\end{subequations}
Now the applications cited in the introduction follow immediately and require no 
further proof other than to cite these results for the respective special cases of $f$. 
We provide other similar 
corollaries and examples of these factorization theorem results below. 

\begin{example}[Sum-of-Divisors Functions] 
For any $\alpha \in \mathbb{C}$, the expansion identity given in 
\eqref{eqn_factthm_exp_idents_v1_eq.III} also implies the following new formula for the 
generalized sum-of-divisors functions, $\sigma_{\alpha} = \sum_{d|n} d^{\alpha}$: 
\[
\sigma_{\alpha}(n) = \sum_{d|n} \sum_{k=1}^{d} t_{d,k}^{(-1)}\left( 
     \sum_{\substack{j \geq 0 \\ k+1-G_j > 0}} (-1)^{\ceiling{\frac{j}{2}}} 
     \phi_{\alpha}(k+1-G_j) - [q^k] (q; q)_{\infty}\right). 
\] 
In particular, when $\alpha := 0$ we obtain the next identity for the divisor function 
$d(n) \equiv \sigma_0(n)$ expanded in terms of Euler's totient function, $\phi(n)$. 
\[
d(n) = \sum_{d|n} \sum_{k=1}^{d} t_{d,k}^{(-1)}\left( 
     \sum_{\substack{j \geq 0 \\ k+1-G_j > 0}} (-1)^{\ceiling{\frac{j}{2}}} 
     \phi(k+1-G_j) - [q^k] (q; q)_{\infty}\right). 
\] 
\end{example} 

\begin{remark} 
There are also numerous noteworthy applications of the expansions of the type I sums from 
this section in the context of exact (and asymptotic) expansions of 
named partition functions. For instance, Rademacher's exact series formula for the 
partition function $p(n)$ involves Dedekind sums implicitly expanded through sums of this 
type. Similarly, an asymptotic approximation for the 
named special function $q(n) = [q^n] (-q; q)_{\infty}$ \cite[A000009]{OEIS} which counts the number of 
partitions of $n$ into distinct parts involves an infinite series over modified Bessel 
functions and nested Kloosterman sums \cite[\S 26.10(vi)]{NISTHB}. 
We have not attempted to study the usefulness of our new finite sums in these contexts in 
deconstructing asymptotic properties in these more famous examples of partition formulas. 
A detailed treatment is nonetheless suggested as an exercise to readers which may unravel some 
undiscovered combinatorial twists to the expansions of such sums. 
\end{remark} 

\begin{example}[Menon's Identity and Related Arithmetical Sums] 
\label{example_Menons_Identity_Toth} 
We can use our new results proved in this section to expand new identities for 
known closed-forms of special arithmetic sums. For example, 
\emph{Menon's identity} \cite{TOTH} states that 
\[
\phi(n) d(n) = \sum_{\substack{1 \leq k \leq n \\ (k,n)=1}} \gcd(k-1, n), 
\]
where $\phi(n)$ is Euler's totient function and $d(n) = \sigma_0(n)$ is the 
divisor function. We can then expand the right-hand-side of Menon's identity as 
follows: 
\[
\phi(n) d(n) = \sum_{j=0}^n \sum_{k=1}^{j-1} \sum_{i=0}^j 
     p(n-j) (-1)^{\ceiling{i/2}} \chi_{1,k}(j-k-G_i) \Iverson{j-k-G_i \geq 1} 
     \gcd(k-1, n). 
\] 
As another application, we show a closely related identity considered by T\'oth in \cite{TOTH}. 
T\'oth's identity states that (\cf \cite{GCD-SUMS}) for an arithmetic function $f$ we have
\[
\sum_{\substack{1 \leq k \leq n \\ (k,n)=1}} f\left(\gcd(k-1, n)\right) = 
     \phi(n) \cdot \sum_{d|n} \frac{(\mu \ast f)(d)}{\phi(d)}. 
\] 
We can use our new formulas to write a gcd-related recurrence relation for $f$ 
in two steps. First, we observe that the right-hand-side divisor sum in the 
previous equation is expanded by 
\begin{align*} 
\sum_{d|n} \frac{(\mu \ast f)(d)}{\phi(d)} & = \frac{1}{\phi(n)} \cdot 
     \sum_{j=0}^n \sum_{k=1}^{j-1} \sum_{i=0}^j 
     p(n-j) (-1)^{\ceiling{i/2}} \chi_{1,k}(j-k-G_i) \Iverson{j-k-G_i \geq 1} 
     f(\gcd(k-1, n)) \\ 
     & \phantom{=\quad\ } + f(1) \Iverson{n = 1}. 
\end{align*} 
The notation $(f \ast g)(n) := \sum_{d|n} f(d) g(n/d)$ denotes the Dirichlet convolution of 
two arithmetic functions, $f$ and $g$. 

Next, by M\"obius inversion and noting that the Dirichlet inverse of $\mu(n)$ is 
$\mu \ast 1 = \varepsilon$, where $\varepsilon(n) = \delta_{n,1}$ is the multiplicative 
identity with respect to Dirichlet convolution, we can express $f(n)$ as follows: 
\begin{align*} 
f(n) & = \sum_{d|n} \sum_{r|d} \sum_{j=0}^r \sum_{k=1}^{j-1} \sum_{i=0}^j \Biggl[
     p(r-j) (-1)^{\ceiling{i/2}} \chi_{1,k}(j-k-G_i) \Iverson{j-k-G_i \geq 1} \times \\ 
     & \phantom{=\sum_{d|n} \sum_{r|d} \sum_{j=0}^r \sum_{k=1}^{j-1} \sum_{i=0}^j\quad\ } 
     \times 
     f(\gcd(k-1, r)) \frac{\phi(d)}{\phi(r)} \mu\left(\frac{d}{r}\right)\Biggr] + 
     f(1) \cdot \sum_{d|n} \phi(d) \mu(d). 
\end{align*} 
\end{example} 

\section{Factorization theorems for sums of the second type} 
\label{Section_Sum_TypeII}

\subsection{Formulas for the inverse matrices} 
\label{subSection_typeII_inv_matrices} 

It happens that in the case of the series expansions we defined in 
\eqref{eqn_intro_two_sum_fact_types_def_v2} of the introduction, the 
corresponding terms of the inverse matrices $u_{n,k}^{(-1)}(f, w)$ 
satisfy considerably simpler formulas that the ordinary matrix entries themselves. 
We first prove a partition-related explicit formula for these inverse matrices as 
Proposition \ref{prop_unk_inverse_matrix} and 
then discuss several applications of this result. 

\begin{proof}[Proof of Proposition \ref{prop_unk_inverse_matrix}]
Let $1 \leq r \leq n$ and for some suitably chosen arithmetic function $g$ define 
\[
\tag{i} 
u_{n,r}^{(-1)}(f, w) := \sum_{m=1}^n L_{f,g,m}(n) w^m. 
\]
By directly expanding the series on the right-hand-side of 
\eqref{eqn_intro_two_sum_fact_types_def_v2}, we obtain that 
\begin{align*} 
g(n) & = \sum_{j=0}^n \left( 
     \sum_{k=1}^j u_{j,k}(f, w) \cdot u_{k,r}^{(-1)}(f, w)\right) p(n-j) \\ 
     & = \sum_{j=0}^n p(n-j) \Iverson{j = r} 
     = p(n-r). 
\end{align*} 
Hence the choice of the function $g$ which satisfies (i) above is given by 
$g(n) := p(n-r)$. The claimed expansion of the inverse matrices then follows. 
\end{proof} 

\begin{prop}
\label{prop_exact_exp_of_Lfgmn} 
We in fact prove the following generalized identity for arbitrary arithmetic functions $f,g$: 
\begin{equation} 
\label{eqn_gen_Lfgmn_formula} 
L_{f,g,m}(n) = \sum_{k=1}^n \sum_{d|(m,n)} f(d) p\left(\frac{n}{d}-k\right) \times 
     \sum_{\substack{j \geq 0 \\ k > G_j}} (-1)^{\ceiling{\frac{j}{2}}} 
     g\left(k-G_j\right).
\end{equation} 
\end{prop}
\begin{proof} 
Since the coefficients on the left-hand-side of the next equation correspond to a 
right-hand-side matrix product as 
\[
[q^n] (q; q)_{\infty} \sum_{m \geq 1} g(m) q^m = 
     \sum_{k=1}^n u_{n,k}(f, w) \sum_{m=1}^k L_{f,g,m}(k) w^m, 
\] 
we can invert the matrix product on the right to obtain that 
\[
\sum_{m=1}^k L_{f,g,m}(k) w^m = \sum_{k=1}^{n} \left(\sum_{m=1}^n \sum_{d|(n,m)} 
     f(d) p\left(\frac{n}{d}-k\right) \cdot w^m\right) \times 
     [q^k] (q; q)_{\infty} \sum_{m \geq 1} g(m), 
\]
so that by comparing coefficients of $w^m$ for $1 \leq m \leq n$, we obtain 
\eqref{eqn_gen_Lfgmn_formula}. 
\end{proof} 

\begin{cor}[A New Formula for Ramanujan Sums]
\label{cor_RamSumNewFormula} 
For any natural numbers $x,m \geq 1$, we have that 
\[
c_x(m) = \sum_{k=1}^x \sum_{d|(m,x)} d \cdot p\left(\frac{x}{d}-k\right) \times 
     \sum_{\substack{j \geq 0 \\ k > G_j}} (-1)^{\ceiling{\frac{j}{2}}} 
     \mu\left(k-G_j\right). 
\] 
\end{cor} 
\begin{proof} 
The Ramanujan sums correspond to the special case of 
Proposition \ref{prop_exact_exp_of_Lfgmn} 
where $f(n) := n$ is the identity function and $g(n) := \mu(n)$ is the 
M\"obius function. 
\end{proof} 

\begin{remark} 
We define the following shorthand notation: 
\begin{equation*} 
\widehat{L}_{f,g}(n; w) := \sum_{m=1}^{n} L_{f,g,m}(n) w^m. 
\end{equation*}
In this notation we have that $u_{n,k}^{(-1)}(f, w) = \widehat{L}_{f,g}(n; w)$ when 
$g(n) := p(n-k)$. 
Moreover, if we denote by $T_n(x)$ the polynomial 
$T_n(x) := 1+x+x^2+\cdots+x^{n-1} = \frac{x^n-1}{x-1}$, 
then we have expansions of these sums as convolved ordinary divisor sums ``twisted'' by 
polynomial terms of the form 
\begin{align}
\label{eqn_lfgnw_sum_divsum_exp_v1} 
\widehat{L}_{f,g}(n; w) & = \sum_{d|n} w^d f(d) T_{n/d}(w^d) g\left(\frac{n}{d}\right) \\ 
\notag 
     & = (w^n-1) \times \sum_{d|n} \frac{w^d}{w^d-1} f(d) g\left(\frac{n}{d}\right). 
\end{align} 
The Dirichlet inverse of these divisor sums is also not difficult to express, though 
we will not give its formula here. 
These sums lead to a first formula for the more challenging expressions for the 
ordinary matrix entries $u_{n,k}(f, w)$ given by the next corollary. 
\end{remark} 

\begin{cor}[A Formula for the Ordinary Matrix Entries] 
\label{cor_unkfw_ord_matrix_formula_v1} 
To distinguish notation, let $\widehat{P}_{f,k}(n; w) := \widehat{L}_{f(n),p(n-k)}(n; w)$, which is an 
immediate shorthand for the matrix inverse terms $u_{n,k}^{(-1)}(f, w)$ that we will precisely 
enumerate below. 
For $n \geq 1$ and $1 \leq k < n$, we have the following formula: 
\begin{align*} 
 & u_{n,k}(f, w) = -\frac{(1-w)^2}{w^2 \cdot (1-w^n)(1-w^k) \cdot f(1)^2}\Biggl( 
     \widehat{P}_{f,k}(n; w) \\ 
     & \phantom{=\ } + \sum_{m=1}^{n-k-1} 
     \left(\frac{w-1}{w f(1)}\right)^{m} \left[ 
     \sum_{k \leq i_1 < \cdots < i_m < n} \frac{\widehat{P}_{f,k}(i_1; w) 
     \widehat{P}_{f,i_1}(i_2; w) \widehat{P}_{f,i_2}(i_3; w) \cdots 
     \widehat{P}_{f,i_{m-1}}(i_m; w) \widehat{P}_{f,i_m}(n; w)}{ 
     (1-w^{i_1}) (1-w^{i_2}) \cdots (1-w^{i_m})} 
     \right]\Biggr) 
\end{align*} 
When $k = n$, we have that 
$$u_{n,n}(f, w) = \frac{1-w}{w(1-w^n) \cdot f(1)}.$$ 
\end{cor} 
\begin{proof}
This follows inductively from the inversion relation between the coefficients of 
a matrix and its inverse. For any invertible lower triangular 
$n \times n$ matrix $(a_{i,j})_{1 \leq i,j \leq n}$, we can express a 
non-recursive formula for the inverse matrix entries as follows: 
\begin{align} 
\label{eqn_lt_invmatrix_formula_v1} 
a_{n,k}^{(-1)} & = \frac{1}{a_{n,n}}\left( 
     -\frac{a_{n,k}}{a_{k,k}} + \sum_{m=1}^{n-k-1} (-1)^{m+1} \left[
     \sum_{k \leq i_1 < \cdots < i_m < n} \frac{a_{i_1,k} a_{i_2,i_1} a_{i_3,i_2} \cdots 
     a_{i_m,i_{m-1}} a_{n,i_m}}{a_{k,k} a_{i_1,i_1} a_{i_2,i_2} \cdots 
     a_{i_m,i_m}}\right] 
     \right) \Iverson{k < n} + \frac{\Iverson{k=n}}{a_{n,n}}. 
\end{align} 
The proof of our result is then just an application of the formula in 
\eqref{eqn_lt_invmatrix_formula_v1} when $a_{n,k} := u_{n,k}^{-1}(f, w)$. 
While the identity in \eqref{eqn_lt_invmatrix_formula_v1} is not immediately obvious 
from the known inversion formulas between inverse matrices in the form of 
\[
a_{n,k}^{(-1)} = \frac{\Iverson{n=k}}{a_{n,n}} - \frac{1}{a_{n,n}} 
     \sum_{j=1}^{n-k-1} a_{n,j} a_{j,k}^{(-1)}, 
\]
the result is easily obtained by induction on $n$ so we do not prove it here. 
\end{proof} 

\subsection{Formulas for simplified variants of the ordinary matrices} 
\label{subSection_formulas_for_ordinary_matrices} 

In Corollary \ref{cor_unkfw_ord_matrix_formula_v1} 
we proved an exact, however somewhat implicit and unsatisfying, 
expansion of the ordinary matrix entries $u_{n,k}(f, w)$ by sums of 
weighted products of the inverse matrices $u_{n,k}^{(-1)}(f, w)$ expressed in 
closed form through Proposition \ref{prop_unk_inverse_matrix}. 
We will now develop the machinery needed to more precisely 
express the ordinary forms of these matrices for 
general cases of the indeterminate indexing parameter $w \in \mathbb{C}$. 

\begin{table}[ht!] 

\begin{equation*} 
\boxed{
\begin{array}{llllll}
 \frac{1}{\widehat{f}(1)} & 0 & 0 & 0 & 0 & 0 \\
 -\frac{\widehat{f}(2)}{\widehat{f}(1)^2}-\frac{1}{\widehat{f}(1)} & \frac{1}{\widehat{f}(1)} & 0 & 0 & 0 & 0 \\
 \frac{\widehat{f}(2)}{\widehat{f}(1)^2}-\frac{\widehat{f}(3)}{\widehat{f}(1)^2}-\frac{1}{\widehat{f}(1)} & -\frac{1}{\widehat{f}(1)} & \frac{1}{\widehat{f}(1)} & 0 &
   0 & 0 \\
 \frac{\widehat{f}(2)^2}{\widehat{f}(1)^3}+\frac{\widehat{f}(2)}{\widehat{f}(1)^2}+\frac{\widehat{f}(3)}{\widehat{f}(1)^2}-\frac{\widehat{f}(4)}{\widehat{f}(1)^2} &
   -\frac{\widehat{f}(2)}{\widehat{f}(1)^2}-\frac{1}{\widehat{f}(1)} & -\frac{1}{\widehat{f}(1)} & \frac{1}{\widehat{f}(1)} & 0 & 0 \\
 -\frac{\widehat{f}(2)^2}{\widehat{f}(1)^3}+\frac{\widehat{f}(3)}{\widehat{f}(1)^2}+\frac{\widehat{f}(4)}{\widehat{f}(1)^2}-\frac{\widehat{f}(5)}{\widehat{f}(1)^2} &
   \frac{\widehat{f}(2)}{\widehat{f}(1)^2} & -\frac{1}{\widehat{f}(1)} & -\frac{1}{\widehat{f}(1)} & \frac{1}{\widehat{f}(1)} & 0 \\
-\frac{\widehat{f}(2)^2}{\widehat{f}(1)^3}+\frac{2 \widehat{f}(3)
   \widehat{f}(2)}{\widehat{f}(1)^3}+\frac{\widehat{f}(4)}{\widehat{f}(1)^2}+\frac{\widehat{f}(5)}{\widehat{f}(1)^2}-\frac{\widehat{f}(6)}{\widehat{f}(1)^2}+\frac{1}{\widehat{f}(1)} &
   \frac{\widehat{f}(2)}{\widehat{f}(1)^2}-\frac{\widehat{f}(3)}{\widehat{f}(1)^2} & 
   -\frac{\widehat{f}(2)}{\widehat{f}(1)^2} & -\frac{1}{\widehat{f}(1)} &
   -\frac{1}{\widehat{f}(1)} & \frac{1}{\widehat{f}(1)}
\end{array}
}
\end{equation*} 

\caption{The simplified matrix entries $\widehat{u}_{n,k}(f, w)$ for $1 \leq n,k \leq 6$ where 
         $\hat{f}(n) = \frac{w^n}{w^n-1} \cdot f(n)$ for arithmetic functions $f$ such that $f(1) \neq 0$.}
\label{table_simplifies_unkfw} 

\end{table}  

\begin{remark}[Simplifications of the Matrix Terms] 
\label{remark_simplified_unkfw} 
Using the formula for the coefficients of $u_{n,k}(f, w)$ in 
\eqref{eqn_intro_two_sum_fact_types_def_v2} expanded by 
\eqref{eqn_lfgnw_sum_divsum_exp_v1}, we can simplify the 
form of the matrix entries we seek closed-form expressions for in the next 
calculations. In particular, we make the following definitions for 
$1 \leq k \leq n$: 
\begin{align*} 
\widehat{f}(n) & := \frac{w^n}{w^n-1} f(n) \\ 
\widehat{u}_{n,k}(f, w) & := (w^k - 1) u_{n,k}(f, w). 
\end{align*} 
Then an equivalent formulation of finding the exact formulas for $u_{n,k}(f, w)$ is 
to find exact expressions expanding the triangular sequence of 
$\widehat{u}_{n,k}(f, w)$ satisfying 
\[
\sum_{\substack{j \geq 0 \\ n-G_j > 0}} (-1)^{\ceiling{\frac{j}{2}}} g(n-G_j) = 
     \sum_{k=1}^{n} \widehat{u}_{n,k}(f, w) \sum_{d|k} 
     \widehat{f}(d) g\left(\frac{n}{d}\right). 
\] 
We will obtain precisely such formulas in the next few results. 
Table \ref{table_simplifies_unkfw} provides the first few rows of our 
simplified matrix entries. 
\end{remark} 

\begin{table}[ht!] 

\centering
\begin{tabular}{|c|l|c|l|c|l|} \hline 
$n$ & $D_f(n)$ & $n$ & $D_f(n)$ & $n$ & $D_f(n)$ \\ \hline 
2 & $-\frac{\widehat{f}(2)}{\widehat{f}(1)^2}$ & 7 & 
    $-\frac{\widehat{f}(7)}{\widehat{f}(1)^2}$ & 12 & 
    $\frac{2 \hat{f}(3) \hat{f}(4)+2
   \hat{f}(2) \hat{f}(6)-\hat{f}(1) \hat{f}(12)}{\hat{f}(1)^3}-\frac{3
   \hat{f}(2)^2
   \hat{f}(3)}{\hat{f}(1)^4}$ \\ 
3 & $-\frac{\widehat{f}(3)}{\widehat{f}(1)^2}$ & 8 & 
    $\frac{2 \widehat{f}(2) \widehat{f}(4)-\widehat{f}(1)
   \widehat{f}(8)}{\widehat{f}(1)^3}-\frac{\widehat{f}(2)^3}{\widehat{f}(1)^4}$ & 13 & 
    $-\frac{\hat{f}(13)}{\hat{f}(1)^2}$ \\ 
4 & $\frac{\widehat{f}(2)^2-\widehat{f}(1)
   \widehat{f}(4)}{\widehat{f}(1)^3}$ & 9 & 
    $\frac{\widehat{f}(3)^2-\widehat{f}(1) \widehat{f}(9)}{\widehat{f}(1)^3}$ & 14 & 
    $\frac{2\hat{f}(2) \hat{f}(7)-\hat{f}(1) \hat{f}(14)}{\hat{f}(1)^3}$ \\ 
5 & $-\frac{\widehat{f}(5)}{\widehat{f}(1)^2}$ & 10 & 
    $\frac{2\widehat{f}(2) \widehat{f}(5)-\widehat{f}(1)
   \widehat{f}(10)}{\widehat{f}(1)^3}$ & 15 & 
    $\frac{2\hat{f}(3) \hat{f}(5)-\hat{f}(1)
   \hat{f}(15)}{\hat{f}(1)^3}$ \\ 
6 & $\frac{2 \widehat{f}(2) \widehat{f}(3)-\widehat{f}(1)
   \widehat{f}(6)}{\widehat{f}(1)^3}$ & 11 & 
    $-\frac{\hat{f}(11)}{\hat{f}(1)^2}$ & 16 & 
    $\frac{\hat{f}(2
   )^4}{\hat{f}(1)^5}-\frac{3 \hat{f}(4)
   \hat{f}(2)^2}{\hat{f}(1)^4}+\frac{\hat{f}(4)^2+2 \hat{f}(2)
   \hat{f}(8)}{\hat{f}(1)^3}-\frac{\hat{f}(16)}{\hat{f}(1)^2}$ \\ \hline 
\end{tabular} 

\bigskip 

\caption{The multiple convolution function $D_f(n)$ for $2 \leq n \leq 16$ where $\widehat{f}(n) := \frac{w^n}{w^n-1} \cdot f(n)$ for an 
         arbitrary arithmetic function $f$ such that $f(1) \neq 0$.} 
\label{table_cases_of_Dn} 

\end{table} 

\begin{definition}[Special Multiple Convolutions] 
\label{def_djn_Dn_func_defs} 
For $n,j \geq 1$, we define the following nested $j$-convolutions of the function $\widehat{f}(n)$ \cite{MERCA-SCHMIDT-RAMJ}: 
\begin{align*} 
\ds_j(f; n) = \begin{cases} 
     (-1)^{\delta_{n,1}} \widehat{f}(n), & \text{ if $j = 1$; } \\ 
     \sum\limits_{\substack{d|n \\ d>1}} \widehat{f}(d) \ds_{j-1}\left(f; \frac{n}{d}\right), & 
     \text{ if $j \geq 2$. } 
     \end{cases} 
\end{align*} 
Then we define our primary multiple convolution function of interest as 
\[
D_f(n) := \sum_{j=1}^n \frac{\ds_{2j}(f; n)}{\widehat{f}(1)^{2j+1}}. 
\]
For example, the first few cases of $D_f(n)$ for $2 \leq n \leq 16$ are 
computed in Table \ref{table_cases_of_Dn}. 
The examples in the table should clarify precisely what multiple convolutions 
we are defining by the function $D_f(n)$. Namely, a signed sum of all possible 
ordinary $k$ Dirichlet convolutions of $\widehat{f}$ with itself evaluated at $n$. 
\end{definition}

\begin{lemma} 
\label{lemma_Dastfhat} 
We claim that for all $n \geq 1$
\[
(D_f \ast \widehat{f})(n) \equiv \sum_{d|n} f(d) D_f(n/d) = 
     -\frac{\widehat{f}(n)}{\widehat{f}(1)} + \varepsilon(n). 
\]
where $\varepsilon(n) \equiv \delta_{n,1}$ is the multiplicative identity function 
with respect to Dirichlet convolution. 
\end{lemma} 
\begin{proof} 
We note that the statement of the lemma is equivalent to showing that 
\begin{equation} 
\label{eqn_fhat_Dinv_formula_exp_v1} 
\left(D_f + \frac{\varepsilon}{\widehat{f}(1)}\right)(n) = \widehat{f}^{-1}(n). 
\end{equation} 
A general recursive formula for the inverse of $\widehat{f}(n)$ is given by 
\cite{APOSTOL-ANUMT} 
\[
\widehat{f}^{-1}(n) = \left(-\frac{1}{\widehat{f}(1)} \sum_{\substack{d|n \\ d > 1}} 
     \widehat{f}(d) \widehat{f}^{-1}(n/d)\right) \Iverson{n > 1} + \frac{1}{\widehat{f}(1)} 
     \Iverson{n = 1}. 
\]
This definition is almost how we defined $\ds_j(f; n)$ above. Let's see how to modify this 
recurrence relation to obtain the formula for $D_f(n)$. 
We can recursively substitute in the formula for $\widehat{f}^{-1}(n)$ until we hit the 
point where successive substitutions only leave the base case of 
$\widehat{f}^{-1}(1) = 1 / \widehat{f}(1)$. This occurs after $\Omega(n)$ substitutions 
where $\Omega(n)$ denotes the number of prime factors of $n$ counting multiplicity. 
We can write the nested formula for $\ds_j(f; n)$ as 
\[
\ds_j(f; n) = \widehat{f}_{\pm} \ast 
\undersetbrace{j-1\text{ factors}}{\left(\widehat{f}-\widehat{f}(1) \varepsilon\right) 
     \ast \cdots 
     \ast \left(\widehat{f}-\widehat{f}(1) \varepsilon\right)}(n), 
\]
where we define 
$\widehat{f}_{\pm}(n) := \widehat{f}(n) \Iverson{n > 1} - \widehat{f}(1) \Iverson{n = 1}$. 
Next, define the nested $k$-convolutions $C_k(n)$ recursively by 
\label{page_eqn_Ckn_kCvls} 
\[
C_k(n) = \begin{cases} 
     \widehat{f}(n) - \widehat{f}(1)\varepsilon(n), & \text{ if $k = 1$; } \\ 
     \sum\limits_{d|n} \left(\widehat{f}(d) - \widehat{f}(1) \varepsilon(d)\right) 
     C_{k-1}(n/d), & \text{ if $k \geq 2$. } 
     \end{cases} 
\] 
Then we can express the inverse of $\widehat{f}(n)$ using this definition as follows: 
\[
\widehat{f}^{-1}(n) = \sum_{d|n} \widehat{f}(d)\left[ 
     \sum_{j=1}^{\Omega(n)} \frac{C_{2k}(n/d)}{\widehat{f}(1)^{\Omega(n)+1}} - 
     \frac{\varepsilon(n/d)}{\widehat{f}(1)^2}\right]. 
\]
Then based on the initial conditions for $k = 1$ (or $j = 1$) in the definitions of 
$C_k(n)$ (or $\ds_j(f; n)$), we see that the function in 
\eqref{eqn_fhat_Dinv_formula_exp_v1} is in fact the inverse of $\widehat{f}(n)$. 
\end{proof} 

\begin{prop} 
\label{prop_exact_pnrec_for_hatunk} 
For all $n \geq 1$ and $1 \leq k \leq n$, we have that 
\[
\sum_{i=0}^{n-1} p(i) \widehat{u}_{n-i,k}(f, w) = D_f\left(\frac{n}{k}\right) 
     \Iverson{n \equiv 0 \bmod{k}} + \frac{1}{\widehat{f}(1)} \Iverson{n = k}. 
\] 
\end{prop} 
\begin{proof} 
We notice that Lemma \ref{lemma_Dastfhat} implies that 
\[
\varepsilon(n) = \left(\left(D_f + \frac{\varepsilon}{\widehat{f}(1)}\right) \ast 
     \widehat{f}\right)(n), 
\] 
where $\varepsilon(n)$ is the multiplicative identity for Dirichlet convolutions. 
The last equation implies that 
\[
\tag{i} 
g(n) = \left(\left(D_f + \frac{\varepsilon}{\widehat{f}(1)}\right) \ast 
     \widehat{f} \ast g\right)(n).  
\] 
Additionally, we know by the expansion of 
\eqref{eqn_intro_two_sum_fact_types_def_v2} and that 
$\widehat{u}_{n,n}(f, w) = 1 / \widehat{f}(1)$ that we also have the expansion 
\[
\tag{ii} 
g(n) = \sum_{k \geq 1} \left[\sum_{j=0}^{n-1} p(j) \widehat{u}_{n-j,k}\right] 
     \sum_{d|k} \widehat{f}(d) g(k/d). 
\] 
So we can equate (i) and (ii) to see that 
\[
\sum_{j=0}^{n-1} p(j) \widehat{u}_{n-j,k} = D_f\left(\frac{n}{k}\right) \Iverson{k | n} + 
     \frac{\Iverson{n=k}}{\widehat{f}(1)}. 
\] 
This establishes our claim. 
\end{proof} 

\begin{cor}[An Exact Formula for the Ordinary Matrix Entries] 
\label{cor_exact_formula_for_hatunkfw} 
For all $n \geq 1$ and $1 \leq k \leq n$, we have that 
\[
\widehat{u}_{n,k}(f, w) = \sum_{\substack{j \geq 0 \\ n-G_j > 0}} 
     (-1)^{\ceiling{\frac{j}{2}}} \left(D_f\left(\frac{n-G_j}{k}\right) 
     \Iverson{n-G_j \equiv 0 \bmod{k}} + \frac{1}{\widehat{f}(1)} 
     \Iverson{n-G_j = k}\right). 
\]
\end{cor} 
\begin{proof} 
This is an immediate consequence of Proposition \ref{prop_exact_pnrec_for_hatunk} 
by noting that the generating function for $p(n)$ is $(q; q)_{\infty}^{-1}$ and that 
\[
(q; q)_{\infty} = \sum_{j \geq 0} (-1)^{\ceiling{\frac{j}{2}}} q^{G_j}. 
     \qedhere
\]
\end{proof} 

\subsection{The general matrices expressed through discrete Fourier transforms} 
\label{subSection_DFT_exps} 

The proof of the result given in 
Theorem \ref{theorem_main_lfgmn_exp_v1} below builds on several key ideas for 
discrete Fourier transforms of the greatest common divisor function 
$(k, n) \equiv \gcd(k, n)$ developed in \cite{kamp-gcd-transform}. 
We adopt the common convention that the function $e(x)$ denotes the 
exponential function $e(x) := e^{2\pi\imath x}$. 
Throughout the remainder of this section we take $k \geq 1$ to be fixed and 
consider divisor sums of the following form which are periodic with respect to $k$: 
\[
L_{f,g,k}(n) := \sum_{d|(n,k)} f(d) g\left(\frac{n}{d}\right). 
\]
In \cite{kamp-gcd-transform} these sums are called $k$-convolutions of $f$ and $g$. 
We will first need to discuss some terminology related to discrete Fourier 
transforms. 

\label{page_DTFT_coefficient_defs}
A \emph{discrete Fourier transform} (DFT) maps a finite sequence of complex numbers 
$\{f[n]\}_{n=0}^{N-1}$ onto their associated Fourier coefficients 
$\{F[n]\}_{n=0}^{N-1}$ defined according to the following reversion 
formulas relating these sequences: 
\begin{align*} 
F[k] & = \phantom{\frac{1}{N}} \sum_{n=0}^{N-1} f[n] e\left(-\frac{kn}{N}\right) \\ 
f[k] & = \frac{1}{N} \sum_{n=0}^{N-1} F[k] e\left(\frac{kn}{N}\right). 
\end{align*} 
The discrete Fourier transform of functions of the greatest common 
divisor, which we will employ 
repeatedly to prove Theorem \ref{theorem_main_lfgmn_exp_v1} below,  
is summarized by the formula in the next lemma \cite{kamp-gcd-transform,SCHRAMM}.

\begin{lemma}[Typical Relations Between Periodic Divisor Sums and Fourier Series]
If we take any two arithmetic functions $f$ and $g$, 
we can express periodic divisor sums modulo any $k \geq 1$ of the form 
\begin{subequations} 
\label{eqn_snk_akm_intro_Fourier_coeff_exps}
\begin{align} 
s_k(f, g; n) & := \sum_{d|(n,k)} f(d) g(k/d) 
     = \sum_{m=1}^k a_k(f, g; m) \cdot e^{2\pi\imath \cdot mn/k}, 
\end{align} 
where the discrete Fourier coefficients in the right-hand-side equation are given by  
\begin{equation}
a_k(f, g; m) = \sum_{d|(m,k)} g(d) f(k/d) \cdot \frac{d}{k}. 
\end{equation} 
\end{subequations}
\end{lemma}
\begin{proof}
For a proof of these relations consult the references 
\cite[\S 8.3]{APOSTOL-ANUMT} \cite[\cf \S 27.10]{NISTHB}. 
These relations are also related to the gcd-transformations proved in 
\cite{kamp-gcd-transform,SCHRAMM}.
\end{proof} 

\begin{notation}
The function $c_m(a)$ defined by 
\[
c_m(a) := \sum_{\substack{k=1 \\ (k, m) = 1}}^{m} e\left(\frac{ka}{m}\right), 
\] 
is Ramanujan's sum. Ramanujan's sum is expanded as in the divisor sums in 
Corollary \ref{cor_RamSumNewFormula} of the last subsection. In the next lemma, the 
convolution operator $\ast$ on two arithmetic functions $f,g$ corresponds to the 
Dirichlet convolution, $f \ast g$. 
\end{notation} 

\begin{lemma}[DFT of Functions of the Greatest Common Divisor] 
Let $h$ be any arithmetic function. 
For natural numbers $m \geq 1$, the discrete Fourier transform (DFT) of $h$ 
is defined by the following function: 
\[
\widehat{h}[a](m) := \sum_{k=1}^m h\left(\gcd(k, m)\right) 
     e\left(\frac{ka}{m}\right). 
\] 
This variant of the DFT of $h(\gcd(n, k))$ (with $k$ a free parameter) satisfies 
$\widehat{h}[a] = h \ast c_{-}(a)$
where the function $\widehat{h}[a]$ is summed explicitly for $n \geq 1$ as the Dirichlet convolution 
\[
\widehat{h}[a](n) = (h \ast c_{-}(a))(n) = \sum_{d|n} h(n/d) c_d(a). 
\]
\end{lemma} 

\begin{definition}[Notation and Special Exponential Sums] 
\label{def_NotationAndSpecialExpSums} 
In what follows, we denote the $\ell^{th}$ Fourier coefficient with respect to $k$ 
of the function $L_{f,g,k}(n)$ by $a_{k,\ell}$ which is well defined since 
$L_{f,g,k}(n) = L_{f,g,k}(n+k)$ is periodic with period $k$. We then have an expansion of this 
function in the form of 
\[
L_{f,g,k}(n) = \sum_{\ell=0}^{k-1} a_{k,\ell} \cdot \e{\frac{\ell n}{k}}, 
\] 
where we can compute these coefficients directly from $L_{f,g,k}(n)$ according to the 
formula 
\[
a_{k,\ell} = \sum_{n=0}^{k-1} L_{f,g,k}(n) \e{-\frac{\ell n}{k}}. 
\] 
We also notice that these Fourier coefficients are given explicitly in terms of 
$f$ and $g$ by the formulas cited in \eqref{eqn_snk_akm_intro_Fourier_coeff_exps} above. 
\end{definition} 

\begin{theorem}
\label{theorem_main_lfgmn_exp_v1} 
For all arithmetic functions $f,g$ and natural numbers $k \geq 1$, we have that 
\begin{equation}
\label{eqn_ThmMainLfgkn_exp}
\sum_{d|k} \sum_{r=0}^{k-1} d \cdot L_{f,g,r}(k) \e{-\frac{rd}{k}} \mu(k/d) = 
  \sum_{d|k} \phi(d) f(d) (k/d)^2 g(k/d),
\end{equation}  
where $\phi(n)$ is Euler's totient function. 
\end{theorem}
\begin{proof} 
We notice that the left-hand-side of \eqref{eqn_ThmMainLfgkn_exp} is a divisor sum 
of the form 
\begin{align*} 
\sum_{d|k} \sum_{r=0}^{k-1} d \cdot L_{f,g,r}(k) \e{-\frac{rd}{k}} \mu(k/d) & = 
     \sum_{d|k} d \cdot a_{k,d} \cdot \mu(k/d), 
\end{align*} 
where the Fourier coefficients in this expansion are given by 
\eqref{eqn_snk_akm_intro_Fourier_coeff_exps} 
\cite[\S 8.3]{APOSTOL-ANUMT} \cite[\S 27.10]{NISTHB}. 
In particular, we have that 
\[
a_{k,d} = k \cdot \sum_{r|(k,d)} g(r) f(k / r) \frac{r}{k}. 
\] 
The left-hand-side of our expansion then becomes 
(\cf \eqref{eqn_sum_over_divsum_interchange_exp_ident} below) 
\begin{align*} 
\sum_{d|k} d \cdot a_{k,d} \mu(k/d) & = 
     \sum_{d|k} \sum_{r|d} r g(r) f\left(\frac{k}{r}\right) \cdot 
     d \mu\left(\frac{k}{d}\right) \\ 
     & = 
     \sum_{d=1}^k \Iverson{d|k} d \cdot \mu\left(\frac{k}{d}\right) \times 
     \sum_{r|d} r g(r) f\left(\frac{k}{r}\right) \\ 
     & = 
     \sum_{r|k} r g(r) f\left(\frac{k}{r}\right) \times 
     \sum_{d=1}^{k/r} dr \cdot \mu\left(\frac{k}{dr}\right) \Iverson{d|k} \\ 
     & = 
     \sum_{r|k} r^2 g(r) \cdot f\left(\frac{k}{r}\right) 
     \phi\left(\frac{k}{r}\right). 
\end{align*} 
We notice that while the exponential sums in the original statement of the 
claim are desirable in expanding applications, this direct expansion is 
difficult to manipulate algebraically. Therefore, we have effectively swapped out the 
exponential sum for the known divisor sum formula for the Fourier coefficients 
implicit in the statement of \eqref{eqn_ThmMainLfgkn_exp} 
in order to prove our key result. 
\end{proof} 

\begin{cor}[An Exact Formula for $g(n)$] 
\label{cor_exact_formula_forn_gn} 
For any $n \geq 1$ and arithmetic functions $f,g$ we have the formula 
$$g(n) = \sum_{d|n} \sum_{j|d} \sum_{r=0}^{d-1} \frac{j \cdot L_{f,g,r}(d)}{d^2} 
  \e{-\frac{rj}{d}} \mu(d/j) y_f(n/d), $$
where $y_f(n) = (\phi f \Id_{-2})^{-1}(n)$ is the Dirichlet inverse of 
$f(n) \phi(n) / n^2$ and $\operatorname{Id}_k(n) := n^k$ for $n \geq 1$. 
\end{cor} 
\begin{proof} 
We first divide both sides of the result in 
Theorem \ref{theorem_main_lfgmn_exp_v1} 
by $k^2$. 
Then we apply a Dirichlet convolution of the left-hand-side of the formula in 
Theorem \ref{theorem_main_lfgmn_exp_v1} with 
$y_f(n)$ defined as above to obtain the exact expansion for $g(n)$. 
\end{proof} 

\begin{cor}[The Mertens Function] 
\label{cor_Mertens_function_v2} 
For all $x \geq 1$, the Mertens function defined in the introduction 
is expanded by Ramanujan's sum as 
\begin{align} 
\label{eqn_Mx_DFT_exp_v1} 
M(x) & = \sum_{d=1}^x \sum_{n=1}^{\floor{\frac{x}{d}}} \sum_{r=0}^{d-1} \left( 
     \sum_{j|d} \frac{j}{d} \cdot \e{-\frac{rj}{d}} \mu\left(\frac{d}{j}\right) 
     \right) \frac{c_d(r)}{d} y(n), 
\end{align} 
where $y(n) = (\phi \Id_{-1})^{-1}(n)$ is the Dirichlet inverse of 
$\phi(n) / n$. 
\end{cor} 
\begin{proof} 
We begin by citing Theorem \ref{theorem_main_lfgmn_exp_v1} 
in the special case corresponding to 
$L_{f,g,k}(n)$ a Ramanujan sum for $f(n) = n$ and $g(n) = \mu(n)$. 
Then we sum over the left-hand-side $g(n)$ in 
Corollary \ref{cor_exact_formula_forn_gn} to obtain the 
initial summation identity for $M(x)$ given by 
\[
\tag{i} 
M(x) = \sum_{n \leq x} \sum_{d|n} \sum_{j|d} 
     \sum_{r=0}^{d-1} \frac{j}{d^2} c_d(r) \e{-\frac{rj}{d}} 
     \mu\left(\frac{d}{j}\right) y\left(\frac{n}{d}\right). 
\]
We can then apply the identity that for any arithmetic functions $h,u,v$ we can 
interchange nested divisor sums as\footnote{ 
     We also have a related identity which allows us to interchange the order of 
     summation in the Anderson-Apostol sums of the following form 
     for any natural numbers $x \geq 1$ and arithmetic functions 
     $f,g,h: \mathbb{N} \rightarrow \mathbb{C}$: 
     $$\sum_{d=1}^x f(d) \sum_{r|(d,x)} g(r) h\left(\frac{d}{r}\right) = 
       \sum_{r|x} g(r) \sum_{d=1}^{x/r} h(d) 
       f\left(\gcd(x,r) d\right).$$ 
} 
\begin{equation} 
\label{eqn_sum_over_divsum_interchange_exp_ident} 
\sum_{k=1}^n \sum_{d|k} h(d) u(k/d) v(k) = \sum_{d=1}^n h(d) \left( 
     \sum_{k=1}^{\floor{\frac{n}{d}}} u(k) v(dk)\right). 
\end{equation} 
Application of this identity to (i) leads to the first form for $M(x)$ stated in 
\eqref{eqn_Mx_DFT_exp_v1}. 
\end{proof} 

\begin{cor}[Euler's Totient Function]
For any $n \geq 1$ we have 
\[
\phi(n) = n \cdot \sum_{d|n} \sum_{j|d} \sum_{r=0}^{d-1} \frac{j}{d^2} c_d(r) 
     \e{-\frac{rj}{d}} \mu\left(\frac{d}{j}\right). 
\]
Additionally, we have the following expansion of the average order sums for $\phi(n)$ 
given by 
\[
\sum_{2 \leq n \leq x} \phi(n) = \sum_{d=1}^x \sum_{r=0}^{d-1} 
     \frac{c_d(r)}{2d} \floor{\frac{x}{d}}\left(\floor{\frac{x}{d}}-1\right) \times 
     \sum_{j|d} j \e{-\frac{rj}{d}} \mu\left(\frac{d}{j}\right). 
\] 
\end{cor} 
\begin{proof} 
We consider the formula in Theorem \ref{theorem_main_lfgmn_exp_v1} with 
$f(n) = n$ and $g(n) = \mu(n)$. Since the Dirichlet inverse of the M\"obius 
function is $\mu \ast 1 = \varepsilon$, we obtain our result by convolution and 
multiplication by the factor of $n$. The average order identity follows from the 
first expansion by applying \eqref{eqn_sum_over_divsum_interchange_exp_ident}. 
\end{proof} 

\subsection{An approach via polynomials and orthogonality relations} 
\label{subSection_orthog_conditions_integral_formulas} 

In Corollary \ref{cor_exact_formula_for_hatunkfw} of 
Section \ref{subSection_formulas_for_ordinary_matrices} 
we proved an exact formula for the modified ordinary matrix entries 
$\widehat{u}_{n,k}(f, w)$ defined by the simplifications of the 
original $u_{n,k}(f, w)$ from \eqref{eqn_intro_two_sum_fact_types_def_v2} in 
Remark \ref{remark_simplified_unkfw} 
(\cf Table \ref{table_simplifies_unkfw}). 
We proved the exact formula for $\widehat{u}_{n,k}(f, w)$ in the previous 
subsection using a more combinatorial argument involving the 
multiple convolutions of the function $D_f(n)$ constructed recursively in 
Definition \ref{def_djn_Dn_func_defs} (see Table \ref{table_cases_of_Dn}). 
In this section we define a sequence of related polynomials 
$P_j(w; t)$ whose coefficients are 
the corresponding simplified forms of the inverse matrices. 
We prove in Proposition \ref{prop_Pjwt_poly_matrix_exp} below that 
\[
\sum_{k=1}^n \widehat{u}_{n,k}(f, w) \cdot P_k(w, t) = t^n. 
\]
where the $P_k(w, t)$ sequence is defined by \eqref{eqn_polys_Pjwt_prop_def_v1}. 
We then find and prove the form of a weight function $\omega(t)$ which 
provides us with the orthogonality condition 
\begin{subequations} 
\begin{align} 
\label{eqn_gen_orthog_condition_stmt_v1} 
\int_{|t|=1,t\in\mathbb{C}} \omega(t) P_i(w; t) P_j(w; t) dt =: \hat{c}_i(w) \Iverson{i=j}, 
\end{align} 
where we define the right-hand-side coefficients by 
\begin{equation} 
\label{eqn_gen_orthog_condition_stmt_v2} 
\hat{c}_n(w) := \int_{|t|=1,t\in\mathbb{C}} \omega(t) \left(P_n(w; t)\right)^2 dt. 
\end{equation} 
This construction, which we develop and make rigorous below, 
provides us with another method by which we may exactly extract the form of the 
simplified matrices $\widehat{u}_{n,k}(f, w)$. 
Namely, we have that for all $n \geq 1$ and $1 \leq k \leq n$ the operation 
\begin{equation} 
\label{eqn_gen_orthog_condition_stmt_v3} 
\widehat{u}_{n,k}(f, w) = \frac{1}{\hat{c}_k(w)} \int_{|t|=1,t\in\mathbb{C}} \omega(t) t^n P_k(w, t) dt, 
\end{equation}  
\end{subequations} 
yields an exact formula for our matrix entries of interest here. 
We now develop the requisite machinery to prove that this construction holds. 

\begin{prop}[A Partition-Related Polynomial Sum] 
\label{prop_Pjwt_poly_matrix_exp} 
Let an arithmetic function $f$ be fixed and for an indeterminate $w \in \mathbb{C}$ 
let $\widehat{f}(n)$ denote 
\[
\widehat{f}(n) := \frac{w^n}{w^n-1} f(n). 
\]
For natural numbers $j \geq 1$ and any indeterminate $w$, let the polynomials 
\begin{equation} 
\label{eqn_polys_Pjwt_prop_def_v1}
P_j(w; t) := \sum_{i=1}^{j} \left(\sum_{d|j} 
     \widehat{f}(d) p\left(\frac{j}{d}-i\right)\right) t^i. 
\end{equation}
Then for all $n \geq 1$ we have that 
\[
\sum_{k=1}^n \widehat{u}_{n,k}(f, w) \cdot P_k(w, t) = t^n. 
\]
\end{prop} 
\begin{proof} 
The claim is equivalent to proving that for each $n \geq 1$, we have that 
\begin{align} 
\label{eqn_Pnwt_unkinvfw_exp_formula_proof_v1} 
(w^n-1) \cdot P_n(w; t) = \sum_{k=1}^n u_{n,k}^{(-1)}(f, w) \cdot t^k. 
\end{align} 
Notice that the previous equation also implies that 
\begin{align} 
\label{eqn_Pnwt_unkinvfw_exp_formula_v2}
P_n(w; t) & = \sum_{k=1}^{n} \widehat{u}_{n,k}^{(-1)}(f, w) \cdot t^k = 
     \sum_{k=1}^n \left(\sum_{d|n} \widehat{f}(d) p\left(\frac{n}{d}-k\right) 
     \right) t^k 
     = \sum_{d|n} \widehat{f}(d) \times 
     \sum_{i=0}^{\frac{n}{d}-1} p(i) \cdot t^{\frac{n}{d}-i}. 
\end{align} 
Now finally, for each $1 \leq k \leq n$, we can expand the coefficients of the left-hand-side as 
\begin{align*} 
[t^k] P_n(w; t) & = \sum_{d|n} \frac{(w^n-1) w^d}{w^d-1} f(d) p(n/d-k) \\ 
     & = \sum_{d|n} f(d) p(n/d-k) \left(\sum_{i=1}^{n/d} w^{id}\right) \\ 
\tag{m = id} 
     & = \sum_{m=1}^{n} \left(\sum_{d|m} f(d) p(n/d-k) \Iverson{d|n}\right) w^m \\ 
     & = \sum_{m=1}^n \left(\sum_{d|(m,n)} f(d) p(n/d-k) \right) w^m. 
\end{align*} 
Hence by the formula for the inverse matrices given in 
Proposition \ref{prop_unk_inverse_matrix}, we have proved our claim. 
\end{proof} 

\begin{prop}[Another Matrix Formula] 
\label{prop_another_matrix_formula} 
For $n \geq 1$ and $1 \leq k \leq n$, we have the following formula for 
the simplified matrix entries: 
\[
\widehat{u}_{n,k}(f, w) = 
     \sum_{\substack{j \geq 0 \\ n-G_j > 0\\k|n-G_j}} 
      (-1)^{\ceiling{\frac{j}{2}}} \cdot  
     \widehat{f}^{-1}\left(\frac{n-G_j}{k}\right). 
\] 
\end{prop} 
\begin{proof} 
According to the last expansion in \eqref{eqn_Pnwt_unkinvfw_exp_formula_v2}, we have that 
\[
\sum_{i=0}^{n-1} p(i) t^{n-i} = \left(\widehat{f}^{-1} \ast P_{-}(w; t)\right)(n), 
\] 
or equivalently that 
\[
t^n = \sum_{\substack{j \geq 0 \\ n-G_j > 0}} (-1)^{\ceiling{\frac{j}{2}}} \cdot 
     \left(\widehat{f}^{-1} \ast P_{-}(w; t)\right)(n-G_j). 
\] 
Then by substituting the previous equation into 
\eqref{eqn_gen_orthog_condition_stmt_v3} we have our result. 
\end{proof} 

\begin{theorem} 
Suppose that the form of the sequence 
$\{\hat{c}_k(\omega)\}_{k \geq 1}$ is given. 
Let $D_{f}(n) := \operatorname{DTFT}[f](n)$ denote the 
discrete time Fourier transform of $f$ at $n$. Then writing 
$t:=e^{iu}$ for $0 \leq u \leq 2\pi$, we have the following exact 
expression for the weight function $\omega(t)$ which varies only depending on 
on the prescribed sequence of $\hat{c}_k(\omega)$: 
\[
\omega\left(e^{\frac{\imath u}{2}}\right) = 2 \cdot 
     D_f\left(\sum_{G_r < i}(-1)^{\left\lceil \frac{r}{2} \right\rceil}
     \sum_{G_l < i - G_r} (-1)^{\left\lceil \frac{l}{2}\right\rceil}
     \left((\hat{c}_{-} \ast \widehat{f}^{-1}) \ast 
     \widehat{f}^{-1}\right)(i-G_l-G_r)\right)(u). 
\]
\end{theorem} 
\begin{proof} 
We have that 
\begin{align*}
\hat{c}_i(\omega) = \sum_{d|i} \widehat{f}(d) \times 
     \sum_{r=0}^{\frac{i}{d}-1} p(r) \times 
     \sum_{c|i}\widehat{f}(c) \times 
     \sum_{l=0}^{\frac{i}{c}-1} p(l) \times 
     \int_{|t|=1} \omega(t) t^{\frac{i}{d}+\frac{i}{c}-r-l} dt. 
\end{align*} 
For the $t:=e^{iu}$ and $0\leq u\leq 2\pi$ defined above, let 
$h(u) = \omega(e^{\imath u}) = \omega(t)$. 
By a direct appeal to M\"oebius inversion we see that 
\begin{align*}
((\hat{c}_{-} \ast \widehat{f}^{-1}) \ast \widehat{f}^{-1})(i) = 
     \sum_{r=0}^{i-1} \sum_{l=0}^{i-1} p(r) p(l) D_{h(u)}^{-1}(2i-l-r).
\end{align*} 
Then we can obtain that 
\begin{align*}
\sum_{r=0}^{i-1} p(r) D_{h}^{-1}(2i-r) = 
     \sum_{G_l < i} (-1)^{\left\lceil \frac{l}{2} \right\rceil} 
     \left((\hat{c}_{-} \ast \widehat{f}^{-1}) \ast 
     \widehat{f}^{-1}\right)(i-G_l), 
\end{align*} 
and 
\begin{align*}
D_{h}^{-1}(2i) = 
     \sum_{G_r < i} (-1)^{\left\lceil \frac{r}{2} \right\rceil} 
     \sum_{G_l<i-G_r}(-1)^{\left\lceil \frac{l}{2} \right\rceil} 
     \left((\hat{c}_{-} \ast \widehat{f}^{-1}) \ast 
     \widehat{f}^{-1}\right)(i-G_l-G_r).
\end{align*}
Thus by taking DTFT of both sides we arrive at the formula 
\begin{align*}
\frac{1}{2} h\left(\frac{u}{2}\right) = 
     D_h\left(\sum_{G_r < i} (-1)^{\left\lceil \frac{r}{2} \right\rceil}
     \sum_{G_l < i-G_r}(-1)^{\left\lceil \frac{l}{2} \right\rceil}
     \left((\hat{c}_{-} \ast \widehat{f}^{-1}) \ast 
     \widehat{f}^{-1}\right)(i-G_l-G_r)\right)(u) 
\end{align*}
Since $h(u) = \omega(e^{\imath u}) = \omega(t)$ this proves our key formula. 
\end{proof} 

\section{Conclusions} 
\label{Section_Concl}

We have proved several new expansions of the type I and type II sums defined by 
\eqref{eqn_sum_vars_TL_defs_intro} for any prescribed arithmetic functions $f$ and $g$. 
Our new results proved in the article include treatments of the expansions of these 
two sum types by both matrix-based factorization theorems and analogous 
identities formulated through discrete Fourier transforms of special function sums. 
The type I sums implicitly define many special number theoretic functions and 
sequences by exponential sum variants of this type. 
Perhaps the most notable canonical example of this sum type is given by 
Euler's totient function which counts the number of integers relatively prime to 
a natural number $n$. The M\"obius function also has a representation in the form of 
a type I sum. 

The type II sums form an alternate flavor of the ordinary 
divisor sums enumerated by Lambert series generating functions and the 
Dirichlet convolutions of two arithmetic functions $f$ and $g$. 
These sums are sometimes refered to as Anderson-Apostol sums, or $k$-convolutions in 
the references. The prototypical example of sums of this type are given by the 
Ramanujan sums $c_q(n)$ which form expansions of many other special number theoretic 
functions by composition and infinite series. 
Our results provide new and useful expansions that characterize common and 
important classes of sums that arise in applications. 
Our results are unique in that we are able to relate partition functions to the 
expansions of these general classes of sums in both cases.

\newpage 

\setcounter{section}{0}
\renewcommand{\thesection}{\Alph{section}}

\section{Appendix: Notation and conventions in the article} 
\label{Appendix_Glossary_NotationConvs}

     \vskip -0.35in
     \printglossary[type={symbols},title={},style={glossstyle},nogroupskip=true]
     \bigskip\hrule\bigskip


\begin{thebibliography}{10} 

\bibitem{APOSTOL-APGRS} 
T. M. Apostol, {Arithmetical properties of generalized Ramanujan sums}, 
  {\em Pacific J. Math.} {\bf 41} (1972). 

\bibitem{APOSTOL-ANUMT} 
T. M. Apostol, {\em Introduction to Analytic Number Theory}, Springer, 1976. 

\bibitem{HARDYWRIGHT}
G.~H. Hardy and E.~M. Wright.
\newblock {\it An Introduction to the Theory of Numbers}.
\newblock Oxford University Press, 2008.

\bibitem{RAMANUJAN} 
G. H. Hardy, {\em Ramanujan: Twelve lectures on subjects suggested by his life and work}, 
  AMS Chelsea Publishing, 1999. 

\bibitem{IKEDA} 
S. Ikeda, et. al. {Sums of products of generalized Ramanujan sums}, 
  {\em J. Integer Seq.} {\bf 19} (2016). 

\bibitem{kamp-gcd-transform} 
van der Kamp, Peter H. 
  On the Fourier transform of the greatest common divisor. 
  No. arXiv: 1201.3139. 2012.

\bibitem{KIUCHI-AVGS} 
I. Kiuchoi, {On sums of averages of generalized Ramanujan sums}, 
  {\em Tokyo J, Math.} {\bf 40} (2017). 

\bibitem{GCD-SUMS} 
I. Kiuchoi and  S. S. Eddin, {On sums of weighted averages of $\gcd$-sum functions}, 
   \url{https://arxiv.org/abs/1801.03647} (2018). 

\bibitem{MERCA-LSFACTTHM} 
M. Merca, {The {L}ambert series factorization theorem}, 
  {\it The Ramanujan Journal}, pp. 1--19 (2017). 

\bibitem{MERCA-SCHMIDT-RAMJ} 
M. Merca and M. D. Schmidt, {Factorization Theorems for Generalized Lambert Series and Applications}, {\it The Ramanujan Journal}, 
  to appear (2018). 
  
\bibitem{MERCA-SCHMIDT-LSFACTTHM} 
M. Merca and M. D. Schmidt, {Generating Special Arithmetic Functions by Lambert Series Factorizations}, {\em Contributions to Discrete Mathematics}, to appear (2018). 

\bibitem{NISTHB}
F.~W.~J. Olver, D.~W. Lozier, R.~F. Boisvert, and C.~W. Clark.
\newblock {\it {NIST} Handbook of Mathematical Functions}.
\newblock Cambridge University Press, 2010.

\bibitem{RIORDAN-COMBIDENTS} 
J. Riordan, {\em Combinatorial Identities}, Wiley, 1968. 

\bibitem{SCHMIDT-SODFORMULAS} 
M. D. Schmidt, {Exact Formulas for the Generalized Sum-of-Divisors Functions}, 
  \url{https://arxiv.org/abs/1705.03488}, submitted (2018). 

\bibitem{AA} 
M. D. Schmidt, {New recurrence relations and matrix equations for arithmetic functions generated by Lambert series}, {\em Acta Arithmetica} {\bf 181} (2018). 

\bibitem{SCHRAMM}
W. Schramm, {The Fourier transform of functions of the greatest common divisors}, 
  {\em INTEGERS} {\bf 8} (2008). 

\bibitem{OEIS}
N. J. A. Sloane, {The Online Encyclopedia of Integer Sequences}, 2017, 
  \url{https://oeis.org/}. 

\bibitem{TOTH} 
L. T\'oth, {Menon’s identity and arithmetical sums representing functions of
   several variables}, \url{https://arxiv.org/abs/1103.5861} (2011). 

\end{thebibliography}
\end{document}